\documentclass{article}
\usepackage[utf8]{inputenc}

\usepackage[a4paper, margin=2.5cm]{geometry}
\usepackage[english]{babel}

\usepackage{url}
\usepackage{verbatim}
\usepackage{listings}
\usepackage{amsmath}
\usepackage{amssymb}
\usepackage{psfrag}
\usepackage{multirow}
\usepackage{amsthm}
\usepackage{makecell}
\usepackage{stmaryrd}
\usepackage{csquotes}
\usepackage{graphicx}
\usepackage{import}
\usepackage{stmaryrd}
\usepackage{xcolor}
\usepackage{pgfplots}
\usepackage{tikz}
\usepgfplotslibrary{groupplots}
\pgfplotsset{compat=1.17}

\usepackage{mathtools}

\newcommand{\N}{\mathbb{N}}
\newcommand{\R}{\mathbb{R}}
\newcommand{\Pol}{\mathbb{P}}

\def\width{\beta}

\newcommand{\mbe}{\mathbb{E}}
\newcommand{\mct}{\mathcal{T}}
\newcommand{\mce}{\mathcal{E}}

\newcommand\restr[2]{\ensuremath{\left.#1\right|_{#2}}}

\def \d {{\rm d}}
\def\dx{\,\d x }

\def\ds{\,\d s }
\def \IR {\mathbb{R}}

\theoremstyle{plain}
\newtheorem{definition}{Definition}[section]
\newtheorem{lemma}[definition]{Lemma}
\newtheorem{theorem}[definition]{Theorem}
\theoremstyle{remark}
\newtheorem{remark}[definition]{Remark}

\makeatletter
\renewcommand*\env@matrix[1][*\c@MaxMatrixCols c]{%
  \hskip -\arraycolsep
  \let\@ifnextchar\new@ifnextchar
  \array{#1}}
\makeatother

\def\env@matrix{\hskip -\arraycolsep
  \let\@ifnextchar\new@ifnextchar
  \array{*\c@MaxMatrixCols c}}

\usepackage{caption}

\date{\today}

\usepackage{graphicx}
\usepackage[shortlabels]{enumitem}

\usepackage{ifthen}

\newboolean{shortintro}
\setboolean{shortintro}{true} 
\newcommand{\IntroChoice}[2]{%
  \ifthenelse{\boolean{shortintro}}{#2}{#1}%
}

\title{DG = FEM + flat elements, Part I: Diffusion}
\author{
    Ji\v{r}\'{i} Szotkowski \thanks{Faculty of Mathematics and Physics, Charles University, Sokolovsk\'{a} 83, 186 75 Praha 8, Czech Republic. Corresponding author kucera@karlin.mff.cuni.cz}
    \and V\'{a}clav Ku\v{c}era \footnotemark[1]
    \and Chi-Wang Shu \thanks{Division of Applied Mathematics, Brown University, Providence, RI 02912, USA}
    \and Antoine Quiriny \thanks{Institute of Mechanics, Materials and Civil Engineering (iMMC), Avenue Georges Lema\^{i}tre 4, 1348 Louvain-la-Neuve, Belgium}
    \and Jonathan Lambrechts\footnotemark[3]
    \and Nicolas Mo\"{e}s\footnotemark[3]
    \and Jean-Fran\c{c}ois Remacle\footnotemark[3]
    \thanks{This project has received funding from the European Research Council (ERC) under the European Union’s Horizon research and innovation program (Grant agreement No. 101 071 255). The Research of Chi-Wang Shu is supported in part by NSF grant DMS-2309249. This work was carried out in part during a visit by V\'{a}clav Ku\v{c}era to the Division of Applied Mathematics at Brown University, whose hospitality and financial support are gratefully acknowledged.}
}

\begin{document}

\maketitle

\newcommand{\BibTeX}{{\scshape Bib}\TeX\xspace}

\begin{abstract}
We establish a simple, rigorous, and easy to implement connection between the classical continuous finite element method (FEM) and the discontinuous Galerkin (DG) method for Poisson's problem. The key idea is to insert a vanishing-thickness layer of `dummy' elements along cell interfaces. By modifying the diffusion coefficient on these elements to be proportional to their thickness, we prove the FEM formulation converges to Babu\v{s}ka-Zl\'{a}mal DG with trapezoidal edge quadrature. The scheme is trivial to implement by (i) a mesh edit that introduces degenerate interface elements and (ii) a single Jacobian threshold in an otherwise unmodified FEM code to handle the degenerate elements via the tempered finite element (TFEM) framework. We provide a rigorous derivation of the resulting TFEM-DG scheme, prove optimal $H^1$ and $L^2$ error estimates, and present numerical experiments in 2D and 3D. The method allows for simple implementation of DG in a FEM code and even adaptive element-by-element switching between FEM and DG with minimal coding effort. The framework is readily extensible, as we will demonstrate in a companion paper dedicated to evolutionary nonlinear first‑order hyperbolic systems.
\end{abstract}


\paragraph*{MSCcodes:}
65N30, 65N12, 35J05

\IntroChoice{
\section{Introduction}
The \emph{finite element method} (FEM) and \emph{discontinuous Galerkin} (DG) method are among the most popular methods for the numerical solution of partial differential equations. Each of these methods has its own set of problems for which it is most suited -- classical continuous FEM is typically used for diffusive or elliptic problems, while DG, with its discontinuous approximations, is more suited for nonlinear first‐order hyperbolic conservation laws. Of course, DG can also be applied, e.g., to Poisson's problem, and FEM can be applied to hyperbolic problems as well. However, these applications bring their own challenges: there is an entire `zoo' of DG formulations for Poisson's problem, each of which might or might not be symmetric, consistent, dual consistent, or even elliptic, cf. \cite{arnold-et-al-unified-analysis-of-dg-for-eliptic-problems}. On the other hand, FEM for hyperbolic conservation laws typically requires some global stabilization strategy akin to streamline diffusion, unlike DG, where the stabilization mechanism is contained in the interface numerical fluxes, possibly with locally applied limiters (see, e.g., overviews in \cite{Shu2013BriefSurvey,ChenShuEntropyStableReview}).

Implementation of the two methods is also fundamentally different: FEM (i.e., continuous Galerkin) uses globally continuous spaces with shared global degrees of freedom and assembles element volume integrals without explicitly using interior face terms. On the other hand, DG uses element-local degrees of freedom and, in addition to element integrals, requires face/interface assembly to evaluate numerical fluxes via traces, jumps, and averages, \cite{ReedHill1973, CockburnShu1998_V, Shu2013BriefSurvey}. For these reasons, combining the two methods into a monolithic code (for example, with FEM in parts of the computational domain and DG in the rest) is not a trivial task. However, such combined codes are useful: consider a fluid-structure interaction problem with the equations of elasticity in $\Omega_1$ and compressible Euler equations in $\Omega_2$. While FEM is suitable and more efficient for the first subproblem, DG would be preferred in the latter case \cite{Ciarlet,Shu2013BriefSurvey}.

The purpose of this paper is to present a simple, rigorous, and easy‑to‑implement framework connecting DG to FEM. Namely, we show how the DG method can be obtained from FEM by inserting a layer of vanishing‑thickness elements at cell interfaces in the mesh and by modifying the diffusion coefficient on these `dummy' elements. We show that the resulting continuous finite element formulation corresponds to a DG scheme. In the present paper, we do so for elliptic problems (demonstrated on Poisson's problem), where FEM with degenerate elements leads to the Babu\v{s}ka-Zl\'{a}mal DG scheme with trapezoidal quadrature on cell interfaces \cite{Babuska-Zlamal, Babuska:1973:FEM_penalty, arnold-et-al-unified-analysis-of-dg-for-eliptic-problems}. In a followup companion paper, we discuss the same ideas for nonlinear first‑order hyperbolic equations, where the degenerate elements naturally give rise to an interface numerical flux. The presented technique allows for a rigorous derivation of the DG scheme as a limit of finite elements, when the thickness $\beta$ of the vanishing‑thickness elements goes to zero. Furthermore, the resulting scheme can be analyzed by standard techniques (optimal $H^1$ and $L^2$ error estimates).

The implementation of the presented technique is trivial within any FEM code. We show how the DG scheme can be equivalently obtained by simple thresholding of (numerically) zero Jacobians $J$ of the flat degenerate interface elements by a nonzero threshold $J_{\min}$ when evaluating element integrals, gradients of basis functions, etc. This modification can be implemented by essentially one line of code in any off‑the‑shelf FEM implementation. Practically, the DG scheme is therefore implemented in a FEM code by:
\begin{enumerate}
    \item a problem‑independent mesh edit that introduces flat zero‑measure elements at cell interfaces,
    \item inserting a simple threshold on (near‑)zero Jacobians into the otherwise unmodified FEM code.
\end{enumerate}
This mechanism (thresholding of near‑zero Jacobians) was introduced in the context of the \emph{Tempered Finite Element Method} (TFEM), \cite{quiriny2024temperedfiniteelementmethod}, as a means to compute accurate FEM solutions on meshes with degenerate elements. It was also demonstrated that zero‑thickness elements along a codimension‑1 curve or surface, with the aforementioned thresholding, correspond to a simple jump‑penalization mortaring scheme. Here we essentially extend the approach to introduce jump penalties at every cell interface. For this reason we will refer to the resulting DG scheme as \emph{TFEM‑DG}. As we will demonstrate (Remark~\ref{rem:WOPSIP}), the TFEM-DG scheme is very close to the WOPSIP (weakly over-penalized symmetric interior penalty) DG method of \cite{WOPSIP}, as both can be viewed as simple quadrature versions of the basic Babu\v{s}ka-Zl\'{a}mal DG scheme and both have the same requirements on the penalization parameter.

We note that the use of degenerate elements in finite element methods is not new. There is a large body of work, so‑called \emph{cohesive elements} or \emph{zero‑thickness interface elements}, which models internal discontinuities by inserting vanishing‑thickness interface elements whose two sides carry independent traces of the primary field and are linked by a traction-separation law. This technique originated in fracture/contact and is now used in multi‑physics settings such as hydro‑mechanical coupling \cite{Dugdale1960,Barenblatt1962,ParkPaulino2011,DurandTrinidade2021,deFranciscoCarol2020,GaroleraEtAl2013,EsfahaniGajo2024}. Our approach is related in spirit -- both frameworks enable inter‑element discontinuities via a mesh‑embedded interface, but differs in purpose and execution. Instead of introducing a dedicated cohesive element with extra interface degrees of freedom, we keep standard continuous FEM on the entire mesh, including flat elements, and recover DG by thresholding of Jacobians. Moreover, cohesive elements are conceptually a modeling device (the interface law is prescribed and discretized), whereas our TFEM‑DG construction is a numerical device that derives the DG interface terms from a geometric limit. Practically, cohesive formulations require explicit interface elements and duplicated trace degrees of freedom, while our method only edits the mesh and leaves the bulk FEM assembly unchanged.

The structure of the paper is as follows. In Section~\ref{sec:1D}, we deal with a simple model problem: Poisson's equation in 1D. In the 1D mesh (partition), we introduce auxiliary small elements of length $\beta$ at every cell interface (mesh node). Furthermore, we modify the diffusion coefficient to be proportional to $\beta$, namely $\beta D$ for some constant $D$. We show how the piecewise linear FEM formulation of such a problem leads, in the limit $\beta \to 0$, to the Babu\v{s}ka-Zl\'{a}mal DG scheme with jump‑penalization parameter $D$. Moreover, we show how the limiting DG scheme can be equivalently implemented by simply thresholding of the element Jacobians by $1/D$ on the zero‑measure elements.

In Section~\ref{sec:2D}, we extend the approach to two dimensions by considering Poisson's problem in 2D. We discuss the geometry of the auxiliary flat elements (of thickness $\beta$) introduced at cell interfaces and prove that in this case FEM converges to Babu\v{s}ka-Zl\'{a}mal DG with trapezoidal quadrature on edges, as $\beta \to 0$.

In Section~\ref{sec:Error_estimates}, we analyze the limiting DG scheme theoretically, proving optimal broken $H^1$ as well as $L^2$ error estimates. The theory also gives a lower bound for the penalization parameter $D$, or equivalently an upper bound on $J_{\min}$.

Finally, in Section~\ref{sec:Num_exp}, we present numerical experiments in 2D and 3D, demonstrating the sharpness of the presented theory, especially the necessary scaling of $D$ with respect to $h$. We also show how the scheme allows us to easily switch between FEM and DG in parts of the computational domain in an element‑by‑element fashion.

The presented framework is flexible. As noted, a companion paper extends these ideas to nonlinear convective problems including the compressible Euler equations. Altogether, the framework allows the easy and problem‑independent incorporation of DG into an existing FEM code, and the local switching between FEM and DG at will, even on an element‑by‑element basis. This choice might be driven by the physics (FEM for elasticity, DG for fluids), or by the user who wants to quickly and easily test the performance of DG vs.\ FEM on his/her problem. Thus, TFEM‑DG elevates the FEM--DG transition from a software challenge to a mesh‑level modeling choice, governed entirely by the user’s needs or the physics of the problem, rather than the implementation. What used to be a strict prohibition -- degenerate elements -- now becomes a tool we can deliberately exploit.
}
{
\section{Introduction}
The \emph{finite element method} (FEM) and \emph{discontinuous Galerkin} (DG) method are among the most widely used methods for the numerical solution of partial differential equations. Each method is naturally suited to different classes of problems: continuous FEM is typically used for diffusive or elliptic problems, while DG, with its discontinuous approximations, is more suited for nonlinear first‐order hyperbolic conservation laws. Both methods can be applied outside their natural settings, but this introduces additional difficulties. For instance, there exists a large `zoo' of DG formulations for Poisson's problem, with varying properties such as symmetry, consistency, or ellipticity \cite{arnold-et-al-unified-analysis-of-dg-for-eliptic-problems}. Conversely, FEM applied to hyperbolic problems typically requires global stabilization strategies, unlike DG where stabilization is naturally embedded in interface fluxes, possibly combined with limiters \cite{Shu2013BriefSurvey,ChenShuEntropyStableReview}.

From an implementation viewpoint, the two methods differ fundamentally. FEM uses globally continuous spaces with shared degrees of freedom and assembles only element volume integrals, whereas DG employs element-local unknowns and requires additional face/interface contributions involving jumps, averages, and numerical fluxes \cite{ReedHill1973,CockburnShu1998_V,Shu2013BriefSurvey}. As a result, combining FEM and DG in a single code is nontrivial. Nevertheless, such combinations are highly desirable, for example in fluid-structure interaction problems where elasticity in $\Omega_1$ is best treated with FEM and compressible Euler equations in $\Omega_2$ with DG \cite{Ciarlet,Shu2013BriefSurvey}.

The purpose of this paper is to present a simple and rigorous framework connecting DG to FEM. We show that DG can be obtained from FEM by inserting vanishing-thickness elements at cell interfaces and modifying the diffusion coefficient on these elements. The resulting continuous FEM formulation converges to a DG scheme. In this work, we focus on elliptic problems (Poisson), where the method leads to the Babu\v{s}ka-Zl\'{a}mal DG scheme with trapezoidal quadrature on interfaces \cite{Babuska-Zlamal,Babuska:1973:FEM_penalty,arnold-et-al-unified-analysis-of-dg-for-eliptic-problems}. A companion paper extends the approach to nonlinear hyperbolic problems. The key idea is that DG arises as the limit $\beta \to 0$, where $\beta$ is the thickness of the interface elements, allowing analysis via standard FEM techniques.

From an implementation perspective, the method is straightforward. DG can be recovered in any FEM code by introducing flat zero-measure interface elements and applying a simple threshold to (near-)zero Jacobians during assembly. This requires only minimal modification of an existing FEM implementation. This mechanism originates from the \emph{Tempered Finite Element Method} (TFEM) \cite{quiriny2024temperedfiniteelementmethod}, where thresholding of degenerate Jacobians was used to obtain stable solutions on meshes with degenerate elements. Here, we extend this idea to enforce jump penalization across all interfaces, leading to what we call \emph{TFEM-DG}. As discussed in Remark~\ref{rem:WOPSIP}, this scheme is closely related to WOPSIP DG \cite{WOPSIP}, with similar penalization requirements.

The use of degenerate elements is not new. Cohesive or zero-thickness interface elements are widely used to model discontinuities, introducing independent traces linked by a traction-separation law \cite{Dugdale1960,Barenblatt1962,ParkPaulino2011,DurandTrinidade2021,deFranciscoCarol2020,GaroleraEtAl2013,EsfahaniGajo2024}. While related in spirit, our approach differs fundamentally: we do not introduce additional interface unknowns, but retain a standard continuous FEM formulation and recover DG through Jacobian thresholding. Thus, cohesive elements serve as a modeling tool, whereas TFEM-DG is a numerical construction deriving DG from a geometric limit.

The paper is organized as follows. In Section~\ref{sec:1D}, we introduce the method on a 1D Poisson problem, inserting elements of size $\beta$ at interfaces and scaling the diffusion coefficient as $\beta D$. We show that, as $\beta \to 0$, the formulation converges to the Babu\v{s}ka-Zl\'{a}mal DG scheme with penalization parameter $D$, and that the same scheme is obtained by Jacobian thresholding. In Section~\ref{sec:2D}, we extend the construction to 2D and show convergence to the DG scheme with trapezoidal quadrature on edges. In Section~\ref{sec:Error_estimates}, we prove optimal broken $H^1$ and $L^2$ error estimates and derive bounds on the penalization parameter. Finally, Section~\ref{sec:Num_exp} presents numerical experiments in 2D and 3D, illustrating the sharpness of the theory and the dependence on $D$ and $h$, as well as the ability to switch locally between FEM and DG.

The framework is flexible and extends to nonlinear convective problems, including the compressible Euler equations. More generally, it enables the incorporation of DG into an existing FEM code in a simple, problem-independent manner, and allows local switching between FEM and DG on an element-by-element basis. What was traditionally viewed as a limitation -- the presence of degenerate elements -- becomes here a deliberate and useful numerical tool.
}

\section{From FEM to DG: 1D motivation}
\label{sec:1D}
As motivation for our approach we consider a simple 1D Poisson problem on an interval $(a,b)\subset\IR$:
\begin{equation}
\label{Cont_prob_strong_1D}
-u^{\prime\prime}=f,\quad \text{ on } (a,b)
\end{equation}
with zero boundary conditions (purely for simplicity). In a DG scheme we take a partition (mesh) of $[a,b]$ into intervals determined by the mesh nodes
\begin{equation}
\label{1D_partition_DG}
a=x_0<x_1<\ldots<x_{N+1}=b,
\end{equation}
which determine the intervals $I_i=[x_i,x_{i+1}]$, $i=0,\ldots,N$. We define the discontinuous piecewise linear space
\begin{equation}
V_h^D=\{v \in L^2([a, b]):v|_{I_i}\in \Pol_1(I_i),\ i=0,\ldots,N\},
\end{equation}
where $\Pol_1(I_i)$ is the space of linear functions on $I_i$. A function $v_h\in V_h^D$ has two values at every interior node $x_i$: $v_h(x_i^\pm)$ and the jump $[v_h]_{x_i}=v_h(x_i^+)-v_h(x_i^-)$. For consistency, we define $[v_h]_{x_0}=v_h(x_0^+)$ and $[v_h]_{x_{N+1}}=-v_h(x_{N+1}^-)$ at the endpoints. Moreover, when applying integration by parts on an interval $I_i$, we will use the standard notation $[v_h]_{x_i}^{x_{i+1}} = v_h(x_{i+1}^-)-v_h(x_i^+)$.

We wish to derive a DG formulation of (\ref{Cont_prob_strong_1D}) as a limit of the standard FEM formulation, where we degenerate certain `dummy' vertex elements. To this end, we enrich the partition (\ref{1D_partition_DG}) with nodes $y_i$:
\begin{equation}
\label{1D_partition}
a=y_0<x_0<y_1<x_1<\ldots<y_{N+1}<x_{N+1}=b,
\end{equation}
where we consider two types of intervals: $I_i=[x_i,y_{i+1}]$ and $J_i=[y_i,x_{i}]$. The idea is that we let each of the intervals $J_i$ degenerate to the single point $x_i$, cf. Figure~\ref{fig:FEM_mesh_1D_Solution}. Specifically, we let $\width\to 0$, where $\width=|J_i|$ for all $i=0,\dots,N+1$.

\begin{figure}
\centering
\includegraphics[width=0.8\textwidth]{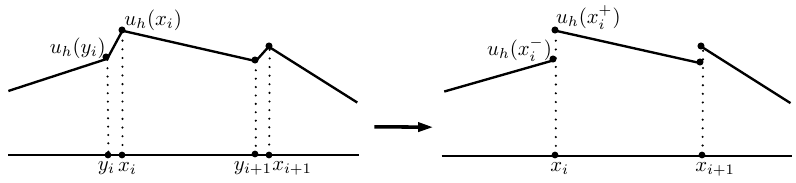}
\caption{FEM solution with auxiliary nodes $y_i$ inserted in the mesh (left) and limiting DG solution (right).}
\label{fig:FEM_mesh_1D_Solution}
\end{figure}

On the mesh (\ref{1D_partition}) we consider the continuous piecewise linear space
\begin{align}\begin{split}
V_h=\{v\in C([a,b]):&v|_{I_i}\in \Pol_1(I_i), i=0,\ldots,N;\\
&v|_{J_i}\in \Pol_1(J_i), i=0,\ldots,N+1; v(a)=v(b)=0\}.
\end{split}\end{align}
Our starting point is a modified version of (\ref{Cont_prob_strong_1D}):
\begin{equation}
\label{Cont_prob_strong_1D_mod}
-\big(\mathcal{D}(x)u^\prime\big)^\prime=f,
\end{equation}
where $\mathcal{D}(x)$ is a modified diffusion coefficient of the form
\begin{equation}
\label{Cont_prob_strong_1D_mod_D}
\mathcal{D}(x)=\begin{cases}
1,\quad &x\in I_i,\ i=0,\ldots,N,\\
{\width}D,\quad &x\in J_i,\ i=0,\ldots,N+1,
\end{cases}
\end{equation}
with a constant $D$ which will be determined appropriately later. We note that we have modified the original diffusion coefficient only on the vanishing elements $J_i$.

The FEM formulation of (\ref{Cont_prob_strong_1D_mod})  reads: Find $u_h\in V_h$ such that 
\begin{equation}
\label{FEM_1D}
\int_a^b \mathcal{D}\, u_h^\prime v_h^\prime\dx=\int_a^b f v_h\dx, \quad\forall v_h\in V_h.
\end{equation}

We now let the size ${\width}$ of the intervals $J_i$ go to zero. We note that since the intervals, basis function, test functions, and numerical solution all depend on $\beta$, we should write e.g. $u_h^\beta$. However, we omit this superscript to make the notation less cluttered, and use this notation only when clarity is essential. 

In the FEM formulation, the piecewise linear solution is determined by its values $u_h(y_i), u_h(x_i)$, while the DG solution is determined by its values $u_h(x_i^-), u_h(x_i^+)$. As we degenerate the auxiliary elements $J_i$ to the single point $x_i$ (e.g. by letting $y_i\to x_i$), we identify the values 
\begin{equation}
\label{convention:1}
u_h(y_i)=u_h(x_i^-),\ u_h(x_i)=u_h(x_i^+)
\end{equation}
in the limit, cf. Figure~\ref{fig:FEM_mesh_1D_Solution}. Using this convention, we can identify any $v_h\in V_h$ with a limiting $v_h\in V_h^D$. It will also be important to note that for any $v_h\in V_h$ 
\begin{equation}
\label{convention:2}
[v_h]_{y_i}^{x_i}\longrightarrow [v_h]_{x_i},
\end{equation}
as ${\width}\to 0$. Using these conventions, we can identify the limiting FEM scheme (\ref{FEM_1D}) for ${\width}\to 0$ with the following DG form. 

\begin{theorem}\label{thm:1D_conv}
The finite element scheme (\ref{FEM_1D}) for ${\width}=0$ is equivalent to the Babu\v{s}ka-Zl\'{a}mal DG method: Find $u_h\in V_h^D$ such that 
\begin{equation}
\label{TFEM_DG_1D}
\sum_{i=0}^{N}\int_{I_i} u_h^\prime v_h^\prime\dx + D\sum_{i=0}^{N+1}[u_h]_{x_i}[v_h]_{x_i}= \int_a^b fv_h\dx,\quad \forall v_h\in V_h^D.
\end{equation}
\end{theorem}

\begin{remark}
The \emph{Babu\v{s}ka-Zl\'{a}mal DG scheme} is a classical method which uses jump penalties to enforce weak continuity at element interfaces. It is a nonconsistent, optimally convergent method, cf. \cite{arnold-et-al-unified-analysis-of-dg-for-eliptic-problems, Babuska-Zlamal, Babuska:1973:FEM_penalty}. 
\end{remark}

\begin{remark}
We note that the intervals $J_i$ have `disappeared' from the limiting formulation, their only remnant being the emergence of the penalization $D[u_h]_{x_i}[v_h]_{x_i}$. 
\end{remark}

\begin{proof}[Proof of Theorem~\ref{thm:1D_conv}]
Let $\varphi_i$, $i=1,\ldots,2N$, be the standard `tent' basis functions of $V_h$ corresponding to all the individual interior nodes $x_i,y_i$. We note that there are no basis functions corresponding to $y_0=a$ and $x_{N+1}=b$ due to the zero Dirichlet boundary conditions at these points. We write the solution of (\ref{FEM_1D}) as $u_h(x)=\sum_{j=1}^{2N} U_j\varphi_j(x)$ and set $v_h:=\varphi_k$ for some $k$. Then (\ref{FEM_1D}) becomes
\begin{equation}
\label{thm:1D_conv:1}
\sum_{j=1}^{2N} U_j\int_a^b \mathcal{D}\,\varphi_j^\prime \varphi_k^\prime \dx= \int_a^b f\varphi_k\dx.
\end{equation}
Consider the left-hand side integral over an element $J_i$. Since $\varphi_j^\prime=[\varphi_j]_{y_i}^{x_i}/{\width}$, we have
\begin{equation}
\label{thm:1D_conv:6}
\int_{J_i} \mathcal{D}\,\varphi_j^\prime \varphi_k^\prime \dx =\int_{J_i} {\width}{D}\frac{1}{{\width}}[\varphi_j]_{y_i}^{x_i} \frac{1}{{\width}}[\varphi_k]_{y_i}^{x_i}\dx ={D}[\varphi_j]_{y_i}^{x_i} [\varphi_k]_{y_i}^{x_i},
\end{equation}
since $|J_i|={\width}$ and all the integrands are constant. Hence
\begin{equation}
    \begin{split}
        \label{thm:1D_conv:7}
\int_a^b \mathcal{D}\,\varphi_j^\prime \varphi_k^\prime \dx &=  \sum_{i=0}^{N+1}\int_{I_i} \varphi_j^\prime \varphi_k^\prime \dx + \sum_{i=0}^{N+1}\int_{J_i} {\width}D\varphi_j^\prime \varphi_k^\prime \dx\\ &= \sum_{i=0}^{N+1}\int_{I_i} \varphi_j^\prime \varphi_k^\prime \dx +D\sum_{i=0}^{N+1}[\varphi_j]_{y_i}^{x_i} [\varphi_k]_{y_i}^{x_i}.
    \end{split}
\end{equation}
We plug this expression into (\ref{thm:1D_conv:1}) and take the limit ${\width}\to 0$. If we apply the convention (\ref{convention:2}), we get the limiting formulation 
\begin{equation}
\label{thm:1D_conv:8}
\sum_{j=1}^{2N} U_j\sum_{i=0}^{N+1}\int_{I_i} \varphi_j^\prime \varphi_k^\prime \dx +D\sum_{j=1}^{2N} U_j\sum_{i=0}^{N+1}[\varphi_j]_{x_i} [\varphi_k]_{x_i}= \int_a^b f\varphi_k\dx,
\end{equation}
which is equivalent to the scheme (\ref{TFEM_DG_1D}).
\end{proof}

\subsection{Implementation}
The actual implementation of the limiting DG scheme from Theorem~\ref{thm:1D_conv} is very easy and can be done by a simple modification (basically one line of code) in any FEM code using the \emph{Tempered Finite Element} (TFEM) approach, cf. \cite{quiriny2024temperedfiniteelementmethod}. This will also be the case in higher dimensions. Thus we can easily implement a DG scheme in a FEM code by simply collapsing certain dummy interface elements and trivially modify the program itself. Here are the details.

Standard FEM implementations compute the elemental entries of a stiffness matrix by mapping onto a reference element, e.g. $(0,1)$ in 1D. Let $J_i$ be a (possibly zero-measure) element, $F:(0,1)\to J_i$ be the mapping from the reference element, and $J$ be the Jacobian of the mapping (obviously $J={\width}$ in 1D). Then    
\begin{equation}
\label{1D_implementation:1}
\int_{J_i} \varphi_j^\prime \varphi_k^\prime \dx =\int_0^1 \frac{\hat{\varphi}_j^\prime}{J} \frac{\hat{\varphi}_k^\prime}{J} J \d \hat{x} = \frac{1}{J}\int_0^1 \hat{\varphi}_j^\prime \hat{\varphi}_k^\prime \d \hat{x}.
\end{equation}
Clearly, problems arise when $J\sim 0$. The idea of TFEM is to replace the problematic Jacobian in the denominator by a threshold $J_{\min}>0$, in our case we take $J_{\min}=1/D$ where $D$ is the constant from (\ref{Cont_prob_strong_1D_mod_D}). If we do this modification in (\ref{1D_implementation:1}), we get 
\begin{equation}
\label{1D_implementation:2}
\frac{1}{J_{\min}}\int_0^1 \hat{\varphi}_j^\prime \hat{\varphi}_k^\prime \d \hat{x} = \frac{J}{J_{\min}}\int_{J_i} \varphi_j^\prime \varphi_k^\prime \d {x}= {\width}D\int_{J_i} \varphi_j^\prime \varphi_k^\prime \d {x}, 
\end{equation}
since $J_{\min}=1/D$ and $J={\width}$. Altogether, we see that introducing a threshold $J_{\min}=1/D$ on the minimal allowed value of the Jacobian $J$ is equivalent to introducing the modified diffusion coefficient ${\width}D$ in the mathematical formulation. Theorem~\ref{thm:main} states that $D\gtrsim h^{-3}$ is necessary, hence the choice $J_{\min}\lesssim h^{3}$ in 1D. 

The situation is essentially the same in $\R^d$, as we shall see in Section~\ref{sec:2D}. The only difference being the optimal choice of $J_{\min}\lesssim h^{d+2}$, cf. Remark~\ref{rem_implementation_2D} on page \pageref{rem_implementation_2D}. Also, we note that in 2D we insert two degenerate elements along every edge of $\mct_h$, cf. Figure~\ref{fig:2D_triangulation_FEM_to_DG}, and three degenerate elements along every face of $\mct_h$ in 3D, cf. Figure~\ref{fig:3D_face_tetrahedra}.  

Altogether, we have the following easy implementation of the Babu\v{s}ka-Zl\'{a}mal scheme from Theorem~\ref{thm:1D_conv}, written here in the case of general $\R^d$.

\medskip
\noindent\fbox{
  \begin{minipage}{0.96\linewidth}
  \label{Implementation}
  \noindent\textbf{Implementation of DG in a FEM code in $\R^d$:}
\begin{enumerate}
    \item Choose a constant \texttt{J\_{min}} $\lesssim h^{d+2}$ (or equivalently $D\gtrsim h^{-3}$).
    \item Take the desired DG mesh and insert $d$ zero-measure elements at every element interface.
    \item Modify the FEM code by prohibiting division by (numerical) zero when computing the element contributions to the stiffness matrix: If the element Jacobian $J$ is too small, threshold it by
    \begin{center}
        \texttt{if (J < \texttt{J\_{min}})\ \  J = \texttt{J\_{min}}}
    \end{center}
    \smallskip
\end{enumerate}
\end{minipage}
}

\smallskip
The resulting FEM code will not only run (even though there are degenerate elements in the mesh), but will be equivalent to a DG implementation! In higher dimensions the recipe is basically the same, although the structure of the zero-measure elements is slightly more complicated.  

In Figure~\ref{fig:1D}, we demonstrate the result of the TFEM-DG approach in 1D. The problem is  $-u^{\prime\prime}=4$ on $[0,1]$. We deliberately choose a very coarse mesh (only 10 elements), so that the discontinuity of the limiting solution is reasonably visible, on finer meshes the DG solution is practically indistinguishable from a continuous FEM solution. We note that we actually plot the whole FEM solution on the mesh augmented by degenerate dummy elements $J_i$ -- this manifests itself as the `vertical' line segments representing the linear solution on the zero-measure elements $J_i$. If we wanted to only plot the true limiting DG solution (without plotting the solution on degenerate elements $J_i$), we would only display the solution on the large elements $I_i$. Less trivial examples are presented in Section~\ref{sec:Num_exp}.

\begin{figure}
    \centering
    \includegraphics[width=0.9\textwidth]{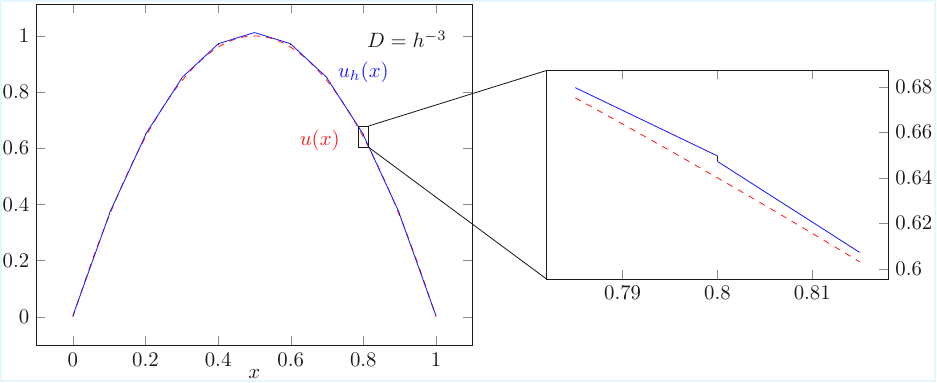}
    \caption{TFEM-DG solution of the problem $-u^{\prime\prime}=4$ on $[0,1]$ with $D=h^{-3}$. Coarse mesh with 10~elements.}
    \label{fig:1D}
\end{figure}

\section{From FEM to DG: 2D}
\label{sec:2D}
Now we will extend the 1D technique to Poisson's problem in 2D. The recipe is essentially the same. To obtain DG on a mesh $\mct_h$, we insert `dummy' elements of thickness $\beta$ along every edge (and possibly also around every vertex), solve a slightly modified problem with diffusion proportional to $\beta$ on the dummy elements, and then degenerate the dummy elements to zero thickness by letting $\beta\to 0$. The limit will also correspond to a DG scheme (Babu\v{s}ka-Zl\'{a}mal with edge quadrature) which will be trivially implementable in a vanilla FEM code by simple thresholding of zero Jacobians. In this section we will derive the exact form of the limiting scheme, for simplicity on a unit square and uniform mesh, however the limiting scheme will be identical on any other mesh. 

Consider $\Omega=[0, 1]^2$, the Sobolev space $V = H^1_0(\Omega)$, and the following problem: Find $u\in V$ such that for all $v \in V$  the following equation holds: 
\begin{align}
\label{rce:slaba-formulace}
    \int_\Omega \nabla u \cdot \nabla v\,dx = \int_\Omega fv \,dx,
\end{align}
where $f \in L^2(\Omega)$ is a given function. This problem is a weak formulation of Poisson's problem with homogeneous Dirichlet boundary conditions on a unit square.

We consider a conforming triangulation $\mct_h$ of $\overline{\Omega}$, i.e. a partition into closed triangles with mutually disjoint interiors whose intersection is either empty, a vertex, or a whole edge. Here the mesh parameter $h=\max_{K\in\mct_h} h_K$ is the largest element diameter. Let $\mbe_h$ denote the set of all edges $\Gamma$ of all elements $K\in\mct_h$. We orient every edge $\Gamma$ by a arbitrarily oriented unit normal $n_\Gamma$. We can then define the two one-sided limits of a piecewise defined function $f:\Omega\to\IR$  at $x\in\Gamma$:
\begin{equation}
f^+(x)=\lim_{t\to 0+}f(x+t n_\Gamma),\quad f^-(x)=\lim_{t\to 0-}f(x+t n_\Gamma)
\end{equation} 
and the jump $[f]_x = f^-(x)-f^+(x)$,
where we will sometimes omit the subscript $x$ for simplicity. Finally, we denote the set of the two endpoints of $\Gamma$ as $e(\Gamma)$.

Let $\mct_h$ be the mesh on which we wish to derive the DG scheme. For now we consider, for simplicity of presentation, a uniform mesh formed by the partition of $\Omega$ into $(1/h)^2$ smaller squares along with their diagonals, cf. Figure~\ref{fig:2D_triangulation_FEM_to_DG}~(\emph{Left}).

\begin{figure}
    \centering
    \includegraphics[width=0.3\linewidth]{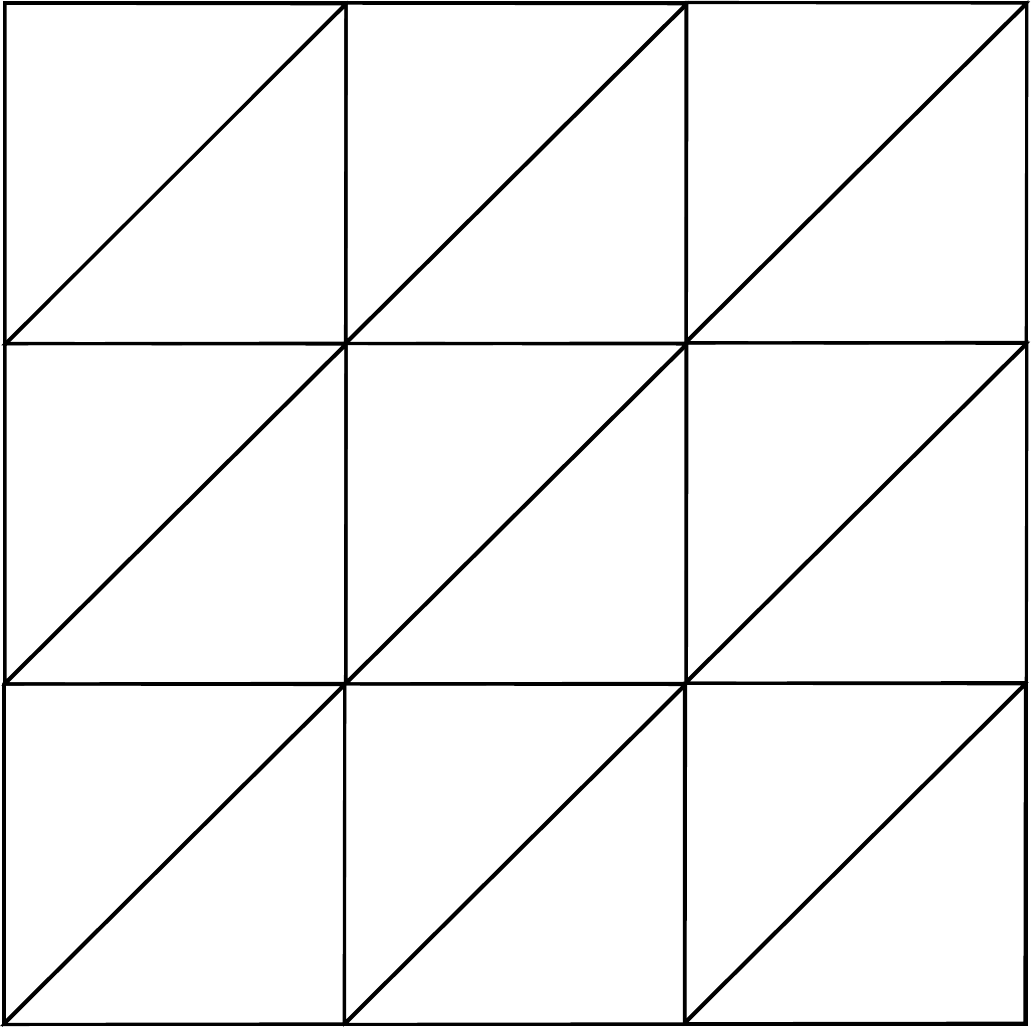}
    \includegraphics[width=0.3\linewidth]{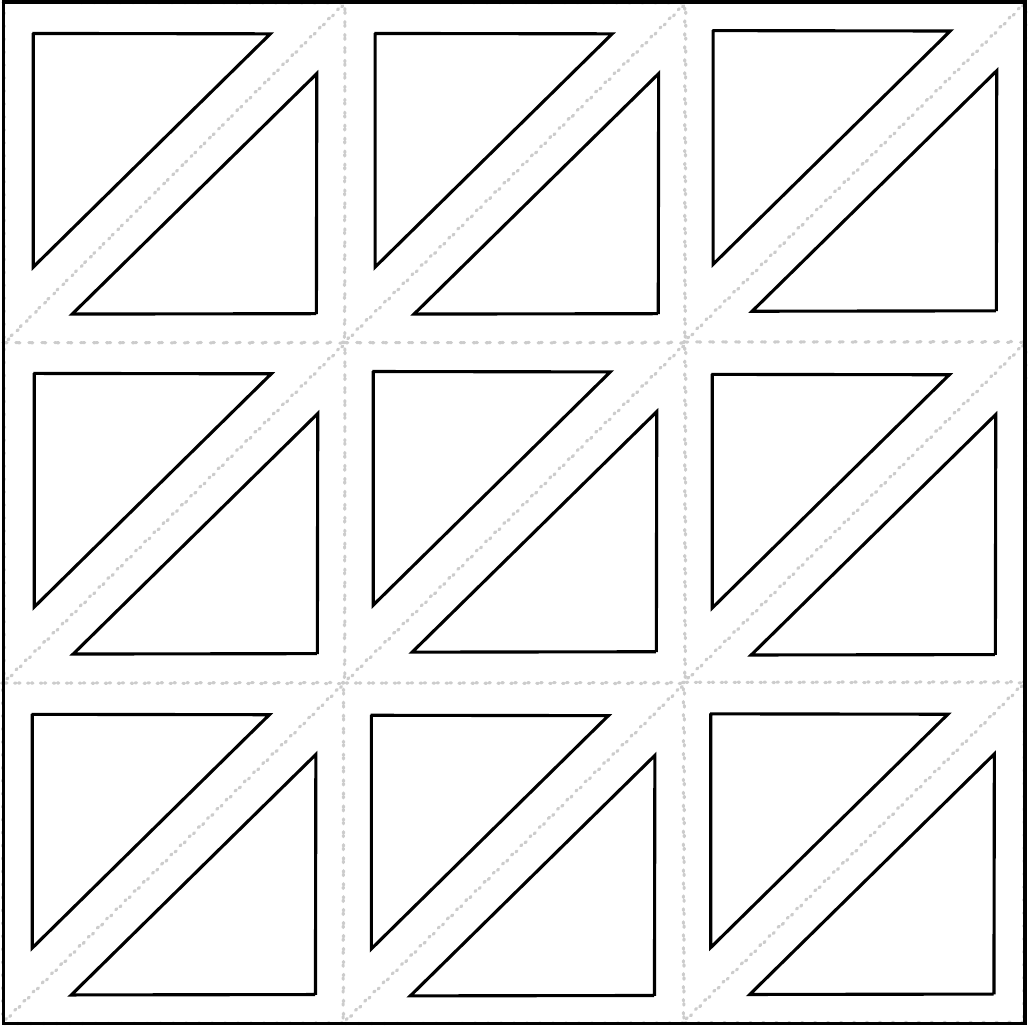}
    \includegraphics[width=0.3\linewidth]{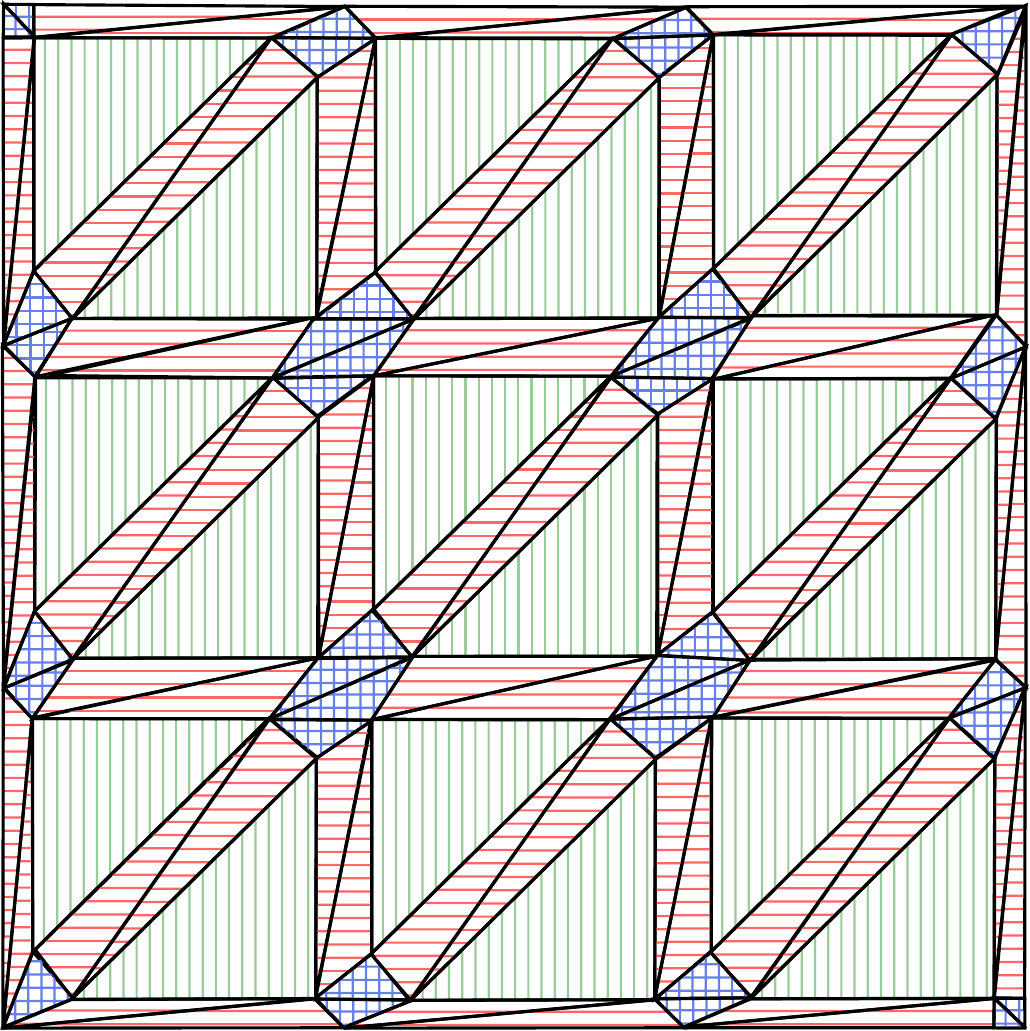}
    \caption{\emph{Left:} Triangulation $\mct_h$ for DG discretization. \emph{Center:} Intermediate `mesh' with shrunken elements. \emph{Right:} Final mesh $\mct_h^\beta$ for equivalent FEM discretization;  $\mce_T$ -- vertical hatch (green), $\mce_E$ -- horizontal hatch (red), $\mce_V$ -- cross‑hatch (blue).}
    \label{fig:2D_triangulation_FEM_to_DG}
\end{figure}





From $\mct_h$ we construct a mesh $\mct_h^\beta$ containing dummy elements of width $\sim\beta$ (sufficiently small) which we will eventually degenerate as $\beta\to 0$. In the first step, we shrink every element $K\in\mct_h$ so that the sides of the new triangles are parallel to the original sides, and the distance from the new to the original sides is $\frac{\beta}{2}, \beta \in \left(0, \beta_{\text{max}}\right)$, where $\beta_{\text{max}} > 0$ is sufficiently small. Note that this creates gaps of width $\beta$ between two adjacent triangles, cf. Figure~\ref{fig:2D_triangulation_FEM_to_DG}~(\emph{Center}).


In the second step, we add triangles to produce a valid mesh, cf.~Figure~\ref{fig:2D_triangulation_FEM_to_DG}~(\emph{Right}). This mesh is obtained by (i) inserting two `dummy' triangular elements along every edge, and (ii) meshing the resulting `holes' adjacent to the original vertexes. In our case the holes are hexagons and can be triangulated in an arbitrary way (without adding any new nodes). We will call this mesh $\mct_h^\beta$. Note that taking the limit $\beta \to 0+$ results in a mesh that looks the same as the original mesh $\mct_h$ in Figure~\ref{fig:2D_triangulation_FEM_to_DG}~(\emph{Left}), but with many `invisible' (i.e. zero-measure) `edge' and `vertex' elements. 

Our goal is to implement DG on $\mct_h$ by taking FEM on $\mct_h^\beta$ and letting $\beta\to 0$. Then the large triangles of $\mct_h^\beta$ correspond (in the limit) to the elements of $\mct_h$. 

\begin{remark}[\bf Implementation]
    This procedure of producing an auxiliary mesh $\mct_h^\beta$ and limiting $\beta\to 0$ is done only for the purpose of deriving the mathematical form of the limiting DG scheme. In practice, one directly constructs a mesh $\mct_h^0$ containing the zero-measure `dummy' edge elements derived from the original mesh $\mct_h$. This is simply a question of supplying the modified mesh topology to the FEM program. 
\end{remark}

\begin{remark}[\bf Number of DOFs]
It is important to note that FEM on the mesh $\mct_h^\beta$ has exactly the same number of DOFs (degrees of freedom) as DG on $\mct_h$! Indeed, the FEM DOFs on $\mct_h^\beta$ are the vertices of the large triangles corresponding to the triangles of $\mct_h$, just as in DG. If $|\mct_h|$ is the number of elements in $\mct_h$, altogether there are $3\times|\mct_h|$ degrees of freedom in FEM on $\mct_h^\beta$, as well as DG on $\mct_h$. This is the point of this specific construction of $\mct_h^\beta$, not to introduce superfluous degrees of freedom. 
\end{remark}

\begin{definition}
    We distinguish three types of elements in $\mct_h^\beta= \mce_T \cup \mce_E \cup \mce_V$, color-coded in Figure~\ref{fig:2D_triangulation_FEM_to_DG}~({Right}) for clarity:
    
    \begin{enumerate}
        \item $\mce_T$: Elements corresponding to the elements of the original mesh $\mct_h$ (vertically hatched, green). The subscript $T$ reflects this fact, as well as the terminology `\emph{Thick elements}'. They have nonzero area even for $\beta=0$. These are the elements with the resulting DG approximation corresponding to $\mct_h$.

        \item $\mce_E$: `\emph{Edge elements}' with height $\beta$ and length $\sim h$ placed along the edges of $\mct_h$ (horizontally hatched, red). These degenerate to edges of $\mct_h$ as $\beta\to 0$.

        \item $\mce_V$: `\emph{Vertex elements}' with diameter $\sim\beta$ placed around the vertices of $\mct_h$ (cross‑hatched, blue). These degenerate to vertexes of $\mct_h$ as $\beta\to 0$.
        
    \end{enumerate}
\end{definition}


\begin{definition}
    Let $h,\beta > 0$. We define the continuous piecewise linear FEM space $V_h\subset V$ on $\mct_h$ as
    $$V_h = \left\{v_h \in C(\overline{\Omega}): \restr{v_h}{K} \in \Pol_1(K), \forall K \in \mct_h, \restr{v_h}{\partial\Omega} = 0\right\},$$
    continuous piecewise linear FEM space $V_h^\beta\subset V$ on $\mct_h^\beta$ as
    $$V_h^\beta = \big\{v_h \in C(\overline{\Omega}): \restr{v_h}{K} \in \Pol_1(K), \forall K \in \mct_h^\beta, \restr{v_h}{\partial \Omega} = 0\big\},$$
    and the discontinuous piecewise linear FEM space $V_h^D$ on $\mct_h$ as 
    $$V_h^D = \left\{v_h \in L^2(\Omega): \restr{v_h}{K} \in\Pol_1(K), \forall K\in\mct_h\right\}.$$
\end{definition}

We now discretize Poisson's problem (\ref{rce:slaba-formulace}) -- first using FEM on $\mct_h^\beta$, which will give an equivalent DG formulation on $\mct_h$ as $\beta\to 0$. Similarly as in the 1D case (\ref{FEM_1D}), we will modify the diffusion coefficient on the elements $\mce_E,\mce_V$ which we will eventually degenerate. We seek a function $u_h \in V_h^\beta$ such that for all $v_h \in V_h^\beta$
\begin{equation}
\label{2D:mod_eq}
\int_\Omega \mathcal{D}\left(x\right) \nabla u_h \cdot \nabla v_h \dx = \int_\Omega fv_h \dx,
\end{equation}
\noindent where the modified diffusion $\mathcal{D}(x)$ is determined by a suitably chosen parameter $D$:
\begin{equation}
    \mathcal{D}(x) =
        \begin{cases}
            1, &  x \in \mce_T, \\
            \beta D, &  x \in \mce_E \cup \mce_V.
    \end{cases}
\end{equation}

We wish to derive the limiting DG scheme to which the FEM scheme (\ref{2D:mod_eq}) reduces as $\beta\to 0$. Let $N =dim(V_h^\beta) = dim(V_h^D)$ and $\varphi_1^\beta, \dots, \varphi_N^\beta$ be standard tent basis functions of $V_h^\beta$. We set the test function $v_h = \varphi_k^\beta$ for some $k\in\{1,\dots, N\}$, and write $u_h$ in terms of coefficients $U_j \in \R, \forall j\in\{1, \dots, N\}$, i.e.
\begin{equation}
u_h = \sum_{j=1}^N U_j\varphi_j^\beta.
\end{equation}
We note that $u_h, \mce_T$, etc. should all have indexes $\beta$, however we will usually omit them to keep the notation less cluttered, except when clarity is of essence.

Now the left-hand side of scheme (\ref{2D:mod_eq}) can be equivalently written as
\begin{equation}
\label{eq:FEM_2D_beta_3_integrals}
\sum_{j=1}^NU_j\bigg(
\underbrace{
    \int_{\mce_T} \nabla\varphi_j^\beta\cdot\nabla\varphi_k^\beta \dx
}_{{(I_1)}} + 
\underbrace{
\int_{\mce_E} \beta D\nabla\varphi_j^\beta\cdot\nabla\varphi_k^\beta \dx
}_{{(I_2)}} + 
\underbrace{
\int_{\mce_V} \beta D\nabla\varphi_j^\beta\cdot\nabla\varphi_k^\beta \dx
}_{{(I_3)}}
\bigg). 
\end{equation}
If we let $\beta\to 0$ and examine the individual terms (for a rigorous derivation see the proof of Theorem~\ref{thm:2D_DG_scheme}), we get:

\begin{itemize}[leftmargin=1em]
    \item $(I_1)$: These terms result in the elemental terms in the DG formulation.
    \item $(I_2)$: These terms give rise to the jump penalty terms in the DG formulation.
    \item $(I_3)$: These terms disappear from the limiting formulation altogether.
\end{itemize}

\begin{theorem}
\label{thm:2D_DG_scheme}
Let $\mct_h$ be a fixed mesh. Let $\mct_h^\beta$ be constructed from $\mct_h$ as described above by inserting edge and vertex dummy elements with thickness parameter $\beta$. Then the FEM formulation (\ref{2D:mod_eq}) on $\mct_h^\beta$ in the limit $\beta\to 0+$ is equivalent to the following DG scheme on $\mct_h$: We seek $u_h\in V_h^D$ such that for all $v_h\in V_h^D$, 
\begin{align}
\label{thm:2D_DG_scheme:BZ-limitni-TFEM-formulace-suma-pres-vrcholy}
\sum_{K\in\mct_h}\int_{K} \nabla u_h\cdot\nabla v_h \dx + 
\frac{D}{2} \sum_{\Gamma \in \mbe_h}\sum_{a\in e(\Gamma)} |\Gamma|[ u_h ]_a[ v_h]_a
=
\int_\Omega fv_h \dx.
\end{align}
\end{theorem}

\begin{remark}[\bf Relation to Babu\v{s}ka-Zl\'{a}mal DG]
The limiting scheme (\ref{thm:2D_DG_scheme:BZ-limitni-TFEM-formulace-suma-pres-vrcholy}) is very close to the classical Babu\v{s}ka-Zl\'{a}mal DG scheme with jump penalties, \cite{arnold-et-al-unified-analysis-of-dg-for-eliptic-problems}:
\begin{equation}
\label{eq:Babuska-Zlamal_2D}
    \sum_{K\in\mct_h}\int_{K} \nabla u_h\cdot\nabla v_h \dx + 
{D} \sum_{\Gamma \in \mbe_h}\int_\Gamma [ u_h ][ v_h]\ds
=
\int_\Omega fv_h \dx,
\end{equation}
for all $v_h\in V_h^D$. In fact, the limiting scheme (\ref{thm:2D_DG_scheme:BZ-limitni-TFEM-formulace-suma-pres-vrcholy}) is simply Babu\v{s}ka-Zl\'{a}mal DG (\ref{eq:Babuska-Zlamal_2D}) with trapezoidal quadrature applied to the edge integrals:
\begin{equation}
    \int_\Gamma [ u_h ][ v_h]\dx\approx \frac{|\Gamma|}{2} \big([ u_h ]_a[ v_h]_a+[ u_h ]_b[ v_h]_b\big),
\end{equation}
where $a,b\in e(\Gamma)$ are the endpoints of $\Gamma$.
\end{remark}

\begin{remark}[\bf Relation to WOPSIP]
\label{rem:WOPSIP}
The limiting scheme (\ref{thm:2D_DG_scheme:BZ-limitni-TFEM-formulace-suma-pres-vrcholy}) is also very similar to the \emph{weakly over-penalized symmetric interior penalty} (WOPSIP) method introduced in \cite{WOPSIP}. This scheme has the form
\begin{align}
\label{eq_WOPSIP}
\sum_{K\in\mct_h}\int_{K} \nabla u_h\cdot\nabla v_h \dx + 
\eta\sum_{\Gamma \in \mbe_h}\frac{1}{|\Gamma|^3}\int_\Gamma \Pi_\Gamma^0[ u_h ]\Pi_\Gamma^0[v_h]\ds
=\int_\Omega fv_h \dx,
\end{align}
where $\Pi_\Gamma^0$ is the orthogonal projection from $L^2(\Gamma)$ onto $\Pol_0(\Gamma)$. 

Similarly as the TFEM-DG scheme (\ref{thm:2D_DG_scheme:BZ-limitni-TFEM-formulace-suma-pres-vrcholy}), WOPSIP can be viewed as a quadrature version of Babu\v{s}ka-Zl\'{a}mal DG (\ref{eq:Babuska-Zlamal_2D}). Namely, when restricted to the discrete space, the edge integrals in (\ref{eq:Babuska-Zlamal_2D}) are approximated by a midpoint quadrature formula in the WOPSIP scheme. Indeed, if $v\in \Pol_1(\Gamma)$, we have $\Pi_\Gamma^0 v= v(\frac{a+b}{2})$, i.e. the value at the midpoint, hence the midpoint formula approximation of the edge integral is
\begin{equation}
\int_\Gamma [ u_h ][ v_h]\dx \approx |\Gamma| [ u_h ]_{\frac{a+b}{2}}[ v_h]_{\frac{a+b}{2}} = \int_\Gamma \Pi_\Gamma^0[ u_h ]\Pi_\Gamma^0[v_h]\ds,
\end{equation}
which is the WOPSIP jump penalization term.

To sum up, the TFEM-DG scheme (\ref{thm:2D_DG_scheme:BZ-limitni-TFEM-formulace-suma-pres-vrcholy}) is equivalent to Babu\v{s}ka-Zl\'{a}mal DG with trapezoidal quadrature, while WOPSIP is equivalent to Babu\v{s}ka-Zl\'{a}mal DG with midpoint quadrature (on the discrete space). Both methods also require the same penalization of (at least) $h^{-3}$, as in (\ref{thm:2D_DG_scheme:BZ-limitni-TFEM-formulace-suma-pres-vrcholy}) and Theorem~\ref{thm:main}.
\end{remark}

\begin{remark}[\bf Implementation]
\label{rem_implementation_2D}
The method (\ref{thm:2D_DG_scheme:BZ-limitni-TFEM-formulace-suma-pres-vrcholy}) is a mathematical interpretation of the limiting scheme. However,  the implementation is much simpler in practice, and follows exactly the TFEM implementation from page  \pageref{Implementation}: Take your FEM code, choose $D$ at least $\sim h^{-3}$, cf. Theorem~\ref{thm:main},  and limit the element Jacobian away from zero by $J_{\min}$ obtained from $D$ on degenerate elements. The only difference in 2D, as opposed to 1D is the relation between $J_{\min}$ and $D$. Similarly as in (\ref{1D_implementation:2}), the introduction of a threshold $J_{\min}$  introduces the factor $J/J_{\min}$ in the element contribution to the mass matrix, cf. \cite{quiriny2024temperedfiniteelementmethod}. In the derivation of the limiting DG scheme, we locally modify the diffusion on edge elements of $\mct_h^\beta$ to be $\beta D$, therefore necessarily
\begin{equation}
\label{rem:implem_2D:1}
    \frac{J}{J_{\min}}=\beta D.
\end{equation}
Now in 2D, we have $J\sim \beta h$ (Jacobian of the mapping onto the reference element), which reduces (\ref{rem:implem_2D:1}) to $J_{\min}\sim\frac{h}{D}$.

In $\R^d$ the same reasoning (volume of `dummy' face simplexes) gives us $J_{\min}\sim h^{d-1}/D$, or $J_{\min}\lesssim h^{d+2}$ due to Theorem~\ref{thm:main}. This has already been incorporated into the implementation `algorithm' on page \pageref{Implementation}.
\end{remark}

\begin{figure}
    \centering
    \includegraphics[width=0.3\linewidth]{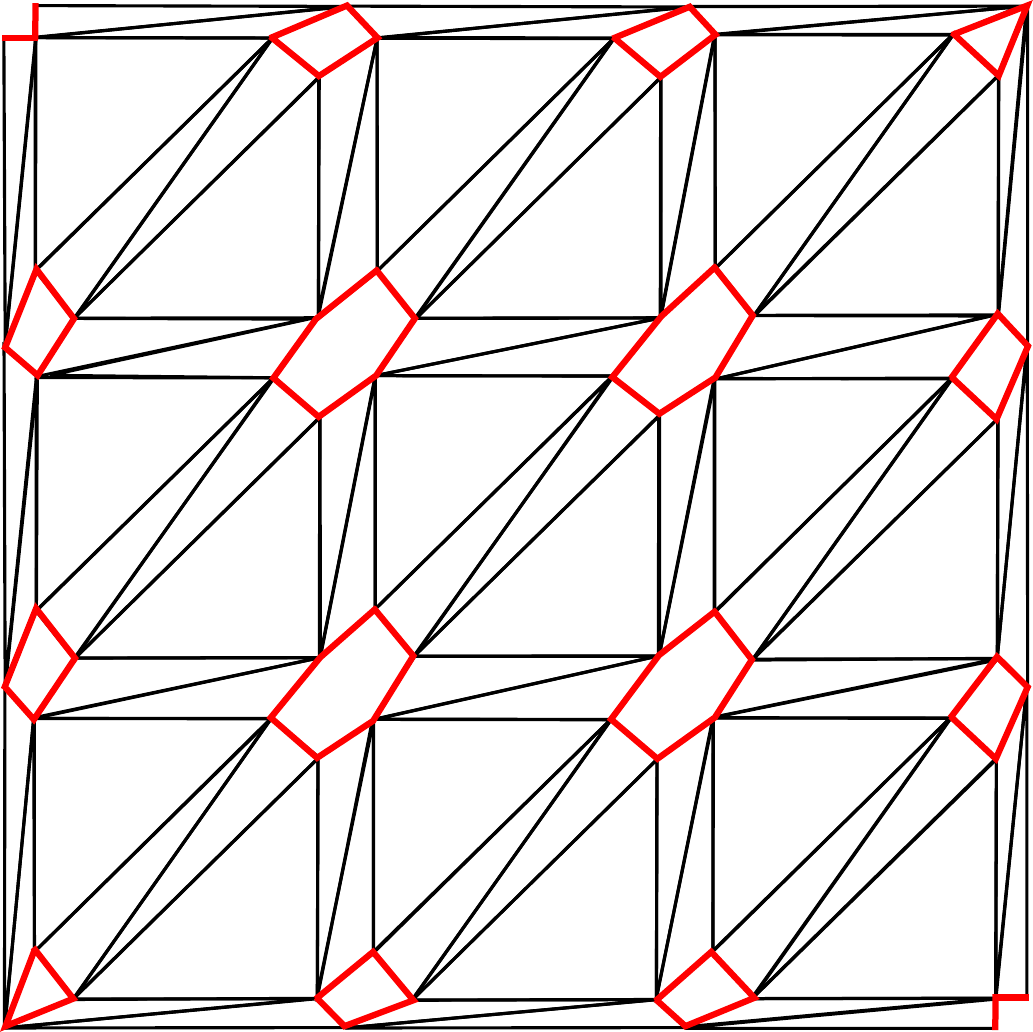}
    \caption{Mesh with vertex elements replaced by `holes' with Neumann boundaries.}
    \label{fig:new-triangulation-neumann}
\end{figure}

\begin{remark}[\bf General diffusion coefficient]
If we consider more generally $-\mathrm{\nabla}\cdot(\alpha(x) \nabla u)$ instead of the basic Laplacian, the analogue of (\ref{2D:mod_eq}) would simply be
\begin{equation}
\nonumber
\int_\Omega \mathcal{D}\left(x\right)\alpha(x) \nabla u_h \cdot \nabla v_h \dx = \int_\Omega fv_h \dx.
\end{equation}
\end{remark}

\begin{remark}[\bf Avoiding vertex elements altogether]
Since the vertex elements $\mce_V$ completely `disappear' from the limiting formulation (\ref{thm:2D_DG_scheme:BZ-limitni-TFEM-formulace-suma-pres-vrcholy}), one can simply avoid them from the start. One can simply \emph{not mesh} the small hexagonal `holes' at the vertices of the original mesh. In our example regular mesh, the resulting triangulation looks like Figure~\ref{fig:new-triangulation-neumann}. Instead of meshing the hexagonal `holes' in the mesh, we apply zero Neumann boundary conditions on their boundaries, i.e. on the red edges in Figure~\ref{fig:new-triangulation-neumann}. Imposing a zero Neumann condition on an edge essentially means that this edge has zero contribution to the resulting scheme, exactly like the vertex elements $\mce_V$. Thus one obtains the same formulation (\ref{thm:2D_DG_scheme:BZ-limitni-TFEM-formulace-suma-pres-vrcholy}) but with less computational effort -- we do not waste resources to compute the zero contributions of vertex elements.
\end{remark}

\begin{remark}[\bf General Dirichlet boundary conditions]
For simplicity we considered only homogeneous Dirichlet boundary conditions in (\ref{rce:slaba-formulace}). However, general Dirichlet conditions are easily incorporated into the scheme. In either case, we introduce a `layer' of degenerate edge elements along $\partial\Omega$, cf. Figures~\ref{fig:2D_triangulation_FEM_to_DG}~and~\ref{fig:new-triangulation-neumann}. If we prescribe the Dirichlet condition $u_D$ on $\partial\Omega$ and incorporate it into the finite element space $V_h^\beta$, then these boundary edge elements will effectively impose the Dirichlet condition by boundary jump penalization, similarly as edge elements in the interior of $\Omega$ impose continuity by penalization.
\end{remark}

\begin{proof}[Proof of Theorem~\ref{thm:2D_DG_scheme}] We distinguish the individual element types:

\textbf{1. Thick elements:} 
We begin with the integral $(I_1)$ from (\ref{eq:FEM_2D_beta_3_integrals}). If $K_T^\beta\in \mce_T$, the integral over $K_T^\beta$ is equal to the corresponding integral of the limiting functions over the limiting element $K\in\mct_h$. To see this, consider that $K_T^\beta$ is simply a homothetic scaling of $K$ by a factor $\lambda_\beta<1$ which satisfies $\lim_{\beta\to 0}\lambda_\beta=1$. If we transform the integral from $K_T^\beta$ to $K$, the area  scales like $\lambda_\beta^2$, while gradients scale like $1/\lambda_\beta$:
    \begin{equation}
    \label{eq:limit_beta_I1}
        \int_{K_T^\beta} \nabla\varphi_j^\beta\cdot\nabla\varphi_k^\beta \dx = \int_{K}  (\lambda_\beta^{-1}\nabla\varphi_j) \cdot(\lambda_\beta^{-1}\nabla\varphi_k) \lambda_\beta^{2} \dx = \int_{K}  \nabla\varphi_j \cdot\nabla\varphi_k \dx.
    \end{equation}
Here it is important that $\varphi_j^\beta$ is a simple transformation of $\varphi_j$ by the homothety. For other operators than the Laplacian, or in 3D, the factors $\lambda_\beta$ will not cancel out and some power of $\lambda_\beta$ will remain in the right-hand side of (\ref{eq:limit_beta_I1}). However, since $\lambda_\beta\to 1$, this factor vanishes in the limit and the integral over  $K_T^\beta$ converges to that over $K$. 

\textbf{2. Vertex elements:} If $K_V^\beta\in \mce_V$ then its area is $|K_V^\beta|\sim\beta^2$, the diffusion coefficient is $\sim\beta$, and the magnitude of the gradients of the basis functions are $|\nabla\varphi_j^\beta|\sim 1/\beta$. Altogether the integral over $K\in\mce_V$ is $\mathcal{O}(\beta)$ and it therefore `disappears' from the limiting DG formulation as $\beta\to 0$:
    \begin{equation}
    \nonumber
        \int_{K_V^\beta} \beta D\nabla\varphi_j^\beta\cdot\nabla\varphi_k^\beta \dx = \mathcal{O}\bigg(\beta^2\beta \frac{1}{\beta}\frac{1}{\beta}\bigg) =\mathcal{O}(\beta)\longrightarrow 0.
    \end{equation}
More rigorously, one can consider that $K^V_\beta$ is a homothetic scaling (shrinking) of $K^V_{\beta_{\max}}$ by a factor $\mu_\beta\sim\beta\to 0$ and apply a scaling argument such as (\ref{eq:limit_beta_I1}). 

\textbf{3. Edge elements:} 
Choose an arbitrary $K_E\in\mce_E$. We label the vertices of $K_E$ as $A, B, C$, as in Figure~\ref{fig:edge-element}. Moreover, denote the foot of the altitude from vertex $A$ to side $BC$ as $A_0$. Now define the orthogonal unit vectors
$$
e_\parallel = \frac{B - C}{|B - C|},\quad e_\perp = \frac{A - A_0}{|A - A_0|},
$$
and denote the projections of $\nabla \varphi$ into these directions as
$$\nabla_\parallel \varphi = \left(\nabla \varphi \cdot e_\parallel\right) e_\parallel,\quad \nabla_\perp \varphi = \left(\nabla \varphi \cdot e_\perp\right) e_\perp.$$

\begin{figure}[!ht]
    \centering
    \includegraphics[width=0.7\textwidth]{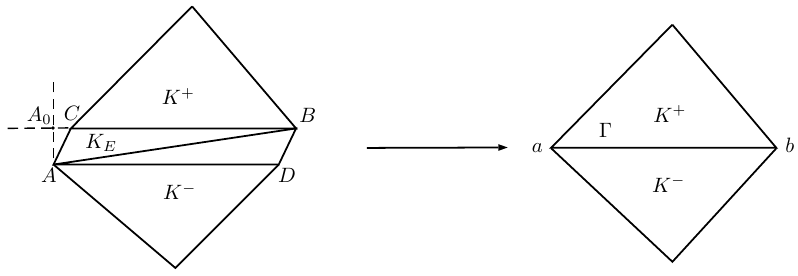}
    \caption{Edge element $K_E$ which degenerates to an edge $\Gamma$.}
    \label{fig:edge-element}
\end{figure}
Since $e_\parallel$ is orthogonal to $e_\perp$, we get (omitting the index $\beta$ for clarity):
$$\int_{K_E}\beta D\nabla\varphi_j\cdot \nabla\varphi_k\,dx = 
\int_{K_E}\beta D\nabla_\parallel\varphi_j\cdot \nabla_\parallel\varphi_k\,dx
+\int_{K_E}\beta D\nabla_\perp\varphi_j\cdot \nabla_\perp\varphi_k\,dx
=(A)+(B).$$
\noindent The integral (A) vanishes as $\beta \to 0+$:
\begin{equation}\nonumber
\begin{split}
\left|\text{(A)}\right| &\leq \int_{K_E}\beta D\left|\nabla_\parallel\varphi_j\right|
\left|\nabla_\parallel\varphi_k\right|\,dx =
\beta D\int_{K_E} \frac{\left|\varphi_j\left(B\right)-\varphi_j\left(C\right)\right|}{|BC|}\frac{\left|\varphi_k\left(B\right)-\varphi_k\left(C\right)\right|}{|BC|}\,dx \\
&\leq\beta D|K_E|\frac{1}{|BC|^2} = \beta D\frac{1}{2}\beta|BC|\frac{1}{|BC|^2}=\frac{D\beta^2}{2|BC|} \longrightarrow 0, \quad \beta \to 0+,
\end{split}
\end{equation}
since $|BC|$ is fixed, independent of $\beta$.

Now we show that the integral (B) gives rise to the DG penalization. Let $K^+, K^-$ be the two `thick' elements in $\mce_T$ neighboring the limiting edge $\Gamma$, to which $K_E$ degenerates, cf. Figure~\ref{fig:edge-element}. Since the point $A_0$ might lie outside the triangle $ABC$ (for obtuse $K_E$), in order to suitably express $\nabla\varphi_j$ we define the extension of a linear function from an element to $\R^2$: For  $K\in\mct_h^\beta$ and ${\varphi} \in\Pol_1(K)$ define $\widetilde{\varphi} \in\Pol_1(\R^2)$ such that $\restr{\widetilde{\varphi}}{K} = \restr{{\varphi}}{K}$. Moreover, let $\varphi_j^\pm =\restr{\varphi_j}{K^\pm}$. Since $\varphi_j$ is continuous, we have $\varphi_j(A) =\varphi_j^-(A)$, and also the extensions stay continuous: $\widetilde{\varphi}_j\left(A_0\right) =\widetilde{\varphi}_j^+\left(A_0\right)$. Thus
\begin{equation}
\nonumber
    \begin{split}
    \text{(B)} &= \beta D\int_{K_E} \frac{{\varphi}_j\left(A\right) - \widetilde{\varphi}_j\left(A_0\right)}{\beta} \frac{{\varphi}_k\left(A\right) - \widetilde{\varphi}_k\left(A_0\right)}{\beta} \,dx \\
&= \beta D\frac{1}{2}\beta|BC|\frac{{\varphi}_j^-\left(A\right) - \widetilde{\varphi}_j^+\left(A_0\right)}{\beta} \frac{{\varphi}_k^-\left(A\right) - \widetilde{\varphi}_k^+\left(A_0\right)}{\beta} \\
&=\frac{D}{2}|BC|\left({\varphi}_j^-\left(A\right) - \widetilde{\varphi}_j^+\left(A_0\right)\right)\left({\varphi}_k^-\left(A\right) - \widetilde{\varphi}_k^+\left(A_0\right)\right)
\xrightarrow{\beta \to 0+} \frac{D}{2}|\Gamma|[\varphi_j]_a[\varphi_k]_a,,        
    \end{split}
\end{equation}
where $a\in e(\Gamma)$ is the endpoint of the limiting edge $\Gamma$ such that $A,A_0,C\to a$. Similarly, by taking the other edge element $ABD$, cf. Figure~\ref{fig:edge-element}, and taking the limit, we get the jump term at the second endpoint $b\in e(\Gamma)$: $\tfrac{1}{2}D|\Gamma|[\varphi_j]_b[\varphi_k]_b$. 

Thus, by summing over all elements $K_E\in \mce_E$, we get the limit of $(I_2)$ from (\ref{eq:FEM_2D_beta_3_integrals}): 
\vspace{-4mm}
\begin{equation}
\nonumber
(I_2)\longrightarrow  
\sum_{\Gamma \in \mbe_h}\sum_{\{a,b\}= e(\Gamma)} \frac{D}{2} |\Gamma|\big([ \varphi_j ]_a[ \varphi_k]_a +[ \varphi_j ]_b[ \varphi_k]_b\big) 
=\sum_{\Gamma \in \mbe_h}\sum_{a\in e(\Gamma)} \frac{D}{2} |\Gamma|[ \varphi_j ]_a[ \varphi_k]_a,
\end{equation}
which after summing over $j$, along with the expression $u_h = \sum_{j=1}^N U_j\varphi_j$, gives the penalization term from (\ref{thm:2D_DG_scheme:BZ-limitni-TFEM-formulace-suma-pres-vrcholy}).
\end{proof}

\section{Error estimates}
\label{sec:Error_estimates}
Now we analyze the limiting DG scheme (\ref{thm:2D_DG_scheme:BZ-limitni-TFEM-formulace-suma-pres-vrcholy}). 


Let $h_{\max} > 0$ and let $I = (0, h_{\max}]$. Let $(\mct_h)_{h\in I}$ be a family of triangulations satisfying \emph{shape-regularity} and \emph{quasi-uniformity}: there exist constants $\kappa>0$ and $c_{qu}>0$ independent of $h$ such that
\begin{equation}
\label{shape-regular}
    h_K\leq \kappa\rho_K,\quad c_{qu}h\leq h_K,
\end{equation}
for all $K\in\mct_h, h\in (0, h_{\max}]$. Here $h_K$ is the maximal edge length of $K\in\mct_h$, $\rho_K$ its inradius, and
$h:=\max_{K\in\mct_h} h_K$ the global mesh size. 


For $k\in\N$ let $H^k(\mct_h) = \left\{v\in L^2(\Omega): \restr{v}{K}\in H^k(K), \forall K\in\mct_h\right\}$ be the \emph{broken Sobolev space} on $\mct_h$ equipped with the seminorm 
$|v|_{H^k(\mct_h)}^2:= \sum_{K\in\mct_h}|v|^2_{H^k(K)}$.
Furthermore, we define the $L^2(\mbe_h)$ and $L^\infty(\mbe_h)$ norms on the edges of $\mct_h$ as 
\begin{align*}
    \|v\|_{L^2(\mbe_h)}^2:= \sum_{\Gamma\in\mbe_h}\|v\|^2_{L^2(\Gamma)}, \quad \|v\|_{L^\infty(\mbe_h)}:= \max_{\Gamma\in\mbe_h}\|v\|_{L^\infty(\Gamma)}.
\end{align*}


We now present the main theorem. We will need $u \in H^2(\Omega) \cap W^{1, \infty}(\mbe_h)$, for which it suffices to have $u \in W^{2,p}(\Omega)$ with $p>2$ in 2D. In the previous sections we  assumed that $\Omega$ is the unit square for simplicity, in the analysis we can assume $\Omega$ to be a convex polygonal domain to ensure regularity of the solution to the dual problem. 

\begin{theorem}
\label{thm:main}
    Let $u \in W^{2,p}(\Omega)$ with $p>2$ be the weak solution satisfying~(\ref{rce:slaba-formulace}). Let $h_{\max} > 0$ and let $I = (0, h_{\max}]$. Let $(\mct_h)_{h\in I}$ be a family of triangulations satisfying the quasi-uniformity and shape-regularity assumptions (\ref{shape-regular}). Let $D\geq c_D/h^3$ for some $c_D$ independent of $h$. Then for $u_h \in V_h^D$, the DG solution  satisfying~(\ref{thm:2D_DG_scheme:BZ-limitni-TFEM-formulace-suma-pres-vrcholy}), there exist constants $C_1,C_2$ independent of $h$, such that the following error estimates hold:
    \begin{align}\label{ieq:final-H1-bound}
        |u-u_h|_{H^1(\mct_h)} + \sqrt{D}\|[ u-u_h ]\|_{L^2(\mbe_h)} &\leq
        C_1h, \\
        \|u-u_h\|_{L^2(\Omega)} &\leq C_2h^2.
    \end{align}
\end{theorem}

\begin{remark}[\bf Local choice of $D$] 
In Theorem~\ref{thm:main}, we choose $D\ge c_Dh^{-3}$ based on the global mesh parameter $h$. In DG penalty schemes it is common to choose the parameter locally as $D\ge c_D|\Gamma|^{-3}$ on edge $\Gamma$. For quasi-uniform meshes this does not make much difference. The analysis is basically identical for the local choice of $D$.
\end{remark}

\begin{remark}[\bf Different exponents in $D$]
The condition on $D$ in the theorem is $D\ge c_Dh^{-3}$. Typically one chooses $D=c_D h^{-p}$ for some power $p$, as is the custom in DG schemes. Theorem~\ref{thm:main} then says that we must choose $p\ge 3$ in order to have optimal $H^1$ and $L^2$ error bounds. Numerical experiments suggest that indeed $p=3$ is the minimal admissible exponent for optimal convergence, cf. Section~\ref{sec:Num_exp}.
\end{remark}

\begin{remark}[\bf General $\R^d$ case]
The proof of Theorem~\ref{thm:main} remains valid in general spatial dimension $\R^d$ assuming quasi-uniform, quasi-regular meshes. The choice of $D\gtrsim h^{-3}$ is independent of $d$.      
\end{remark}

Before we can prove Theorem~\ref{thm:main}, we need some auxiliary results.

\begin{lemma}\label{lem:odhad-polynomu-na-hrane}
    Let $\Gamma\in\mbe_h$ be an edge with endpoints $a,b$, and let $f \in \Pol^1(\Gamma)$ be a first degree polynomial on $\Gamma$. Then 
    $$\int_\Gamma f^2(x)\ds \leq |\Gamma|\frac{f^2(a)+f^2(b)}{2}.$$
\end{lemma}
\begin{proof} The inequality follows either from a simple direct calculation or by using the fact that $f^2$ is convex and applying the Hermite–Hadamard inequality.
\end{proof}

\begin{lemma}[Interpolation error, \cite{Ciarlet}]\label{lem:MAC}
    Let $K\in\mct_h$ be a shape-regular element in a quasi-uniform mesh $\mct_h$. Let $u\in H^2(K)$ and $\Pi_Ku$ be the \emph{linear Lagrange interpolation} of $u$ on $K$. Then there exists a constant $C_I$ independent of $h, u$ such that 
    \begin{align}\label{rce:MAC}
        |u - \Pi_Ku|_{H^1(K)} \leq C_Ih|u|_{H^2(K)}.
    \end{align}
\end{lemma}

\begin{lemma}[Trace inequality]
\label{lem:trace:ineq}
    Let $K \in \mct_h$ be an element and $\Gamma$ one of its sides. Then there exists a constant $C_{Tr} >0$ that depends on $\kappa$, such that for all $v \in H^2(K)$:
    \begin{align*}
    \bigg\|\frac{\partial v}{\partial n_\Gamma}\bigg\|^2_{L^2(\Gamma)} \leq C^2_{Tr}\big(|\Gamma|^{-1}|v|^2_{H^1(K)} + |\Gamma||v|^2_{H^2(K)}\big) \leq \frac{C^2_{Tr}\kappa}{2c_{qu}} h^{-1}\|v\|^2_{H^2(K)}.
\end{align*}
\end{lemma}
\begin{proof}
    The first inequality can be found e.g. in \cite{arnold-et-al-unified-analysis-of-dg-for-eliptic-problems} (inequality (4.7)). To get the second inequality, we note that from elementary geometry of the triangle $K$, we have $2\rho_K\leq|\Gamma|$. Therefore,
    \begin{equation}
    \nonumber
        \frac{1}{|\Gamma|}\leq\frac{1}{2\rho_K}\leq \frac{\kappa}{2h_K}\leq \frac{\kappa}{2c_{qu}h},
    \end{equation}
 and thus
\begin{equation}
\nonumber
    C^2_{Tr}\big(|\Gamma|^{-1}|v|^2_{H^1(K)} + |\Gamma||v|^2_{H^2(K)}\big)\leq C^2_{Tr}|\Gamma|^{-1}\|v\|^2_{H^2(K)}\leq C^2_{Tr} \frac{\kappa}{2c_{qu}h}\|v\|^2_{H^2(K)}.
\end{equation}
\end{proof}

\begin{proof}[Proof of Theorem~\ref{thm:main}, $H^1$-estimates]
    In (\ref{thm:2D_DG_scheme:BZ-limitni-TFEM-formulace-suma-pres-vrcholy}), we have discontinuous test functions,  which we cannot simply insert into the weak form (\ref{rce:slaba-formulace}), we get a consistency error. Since $u \in W^{2,p}(\Omega), p>2$, the solution $u$ satisfies the strong form
\begin{align}\label{rce:silna-formulace-laplace}
    -\Delta u = f,\quad \text{ a.e. in } \Omega
\end{align}
and that the traces $\nabla u \cdot n \in L^\infty(\Gamma)$ on edges $\Gamma$. Multiplying (\ref{rce:silna-formulace-laplace}) by a test function $v_h \in V_h^D$, integrating over $K\in\mct_h$, using Green's theorem, and summing over all $K\in\mct_h$, we conclude that $u$ satisfies:
\begin{align}
\label{rce:presne-reseni-testovane-funkci-z-Vh}
    \sum_{K\in\mct_h}\int_{K} \nabla u \cdot \nabla v_h \dx - \sum_{\Gamma\in \mbe_h}\int_{\Gamma} (\nabla u \cdot n_\Gamma) [ v_h ] \dx = \int_\Omega fv_h \dx.
\end{align}
Subtracting (\ref{thm:2D_DG_scheme:BZ-limitni-TFEM-formulace-suma-pres-vrcholy}) from (\ref{rce:presne-reseni-testovane-funkci-z-Vh}) and using $[u_h]=[u_h-u]$ (continuity of $u$), we get:  
\begin{align}\label{rce:rozdil-presneho-a-priblizneho-reseni-prvni-krok}
\begin{split}    
    0 = \sum_{K\in\mct_h}\int_{K} \left(\nabla u - \nabla u_h\right) \cdot \nabla v_h \dx - \sum_{\Gamma\in \mbe_h}\int_{\Gamma} (\nabla u \cdot n_\Gamma) [ v_h ]\ds \\+ \frac{D}{2}\sum_{\Gamma\in\mbe_h}\sum_{a\in e(\Gamma)} |\Gamma|[ u - u_h ]_a [ v_h ]_a.
\end{split}
\end{align}
Denote $\xi_h = \Pi_hu-u_h$, where $\Pi_hu\in V_h$ is the continuous piecewise linear Lagrange interpolation of $u$ on $\mct_h$. Set $v_h := \xi_h\in V_h^D$. Introduce $\Pi_hu$ and $\nabla\Pi_hu$ in (\ref{rce:rozdil-presneho-a-priblizneho-reseni-prvni-krok}):
\begin{equation}
\nonumber
    \begin{split}
        0 = &\sum_{K\in\mct_h}\int_{K} \Big(\nabla u - \nabla\Pi_hu + \underbrace{\nabla \Pi_hu - \nabla u_h}_{\nabla\xi_h}\Big) \cdot \nabla \xi_h \dx - \sum_{\Gamma\in \mbe_h}\int_{\Gamma} \nabla u \cdot n_\Gamma [ \xi_h ] \ds \\&+
        \frac{D}{2}\sum_{\Gamma\in\mbe_h}\sum_{a\in e(\Gamma)}|\Gamma| [ u - \Pi_hu + \underbrace{\Pi_hu - u_h}_{\xi_h} ]_a [ \xi_h ]_a,
    \end{split}
\end{equation}
which is equivalent to (the continuity of $u-\Pi_hu$ implies $|\Gamma|[ u-\Pi_hu ]_a[ \xi_h ]_a=0$):
\begin{multline}\label{rce:rozdil-presneho-a-priblizneho-reseni-treti-krok}
    |\xi_h|_{H^1(\mct_h)}^2 +
    \frac{D}{2}\sum_{\Gamma\in\mbe_h}\sum_{a\in e(\Gamma)}|\Gamma| [ \xi_h ]^2_a \\=
    \sum_{\Gamma\in\mbe_h}\int_\Gamma (\nabla u\cdot n_\Gamma)[ \xi_h ] \ds  
    - \sum_{K\in\mct_h}\int_{K}(\nabla u - \nabla \Pi_h u) \cdot \nabla\xi_h\dx =: (a) +(b),
\end{multline}

\textbf{Estimate for $\mathbf{(a)}$:} Using the Cauchy-Schwarz inequality, Young's inequality $ab \leq \frac{1}{2D}a^2+\frac{D}{2}b^2$, and Lemma~\ref{lem:odhad-polynomu-na-hrane} respectively:
\begin{align*}
    \text{(a)} &\leq \sum_{\Gamma\in\mbe_h}\int_\Gamma |\nabla u|\big|[ \xi_h ]\big| \ds \leq
    \sum_{\Gamma\in\mbe_h}\int_\Gamma \frac{1}{2D}|\nabla u|^2 + \frac{D}{2}[ \xi_h ]^2 ds \\
    &\leq \sum_{\Gamma\in\mbe_h}\frac{|\Gamma|}{2D}\|\nabla u\|_{L^\infty(\Gamma)}^2 + \sum_{\substack{\Gamma\in\mbe_h\\ \Gamma=(a,b)}}\frac{D}{2}\frac{|\Gamma|}{2}\big([ 
\xi_h ]_a^2 + [ \xi_h ]_b^2\big)\\ 
    &= \frac{1}{2D}\|\nabla u\|_{L^\infty(\mbe)}^2
    \sum_{\Gamma\in\mbe_h}|\Gamma| + \frac{D}{4}\sum_{\Gamma\in\mbe_h}\sum_{a\in e(\Gamma)}|\Gamma|[ \xi_h ]_a^2.
\end{align*}
Now we bound $\sum_\Gamma|\Gamma|$. We use the basic triangle identity $|K|=\tfrac{1}{2}\rho_K|\partial K|$ (area = inradius $\times$ semiperimeter) and the shape-regularity and quasi-uniformity (\ref{shape-regular}):
\begin{align}
\sum_{\Gamma\in\mbe_h}|\Gamma|
\le\sum_{K\in\mct_h}|\partial K|
=\sum_{K\in\mct_h}\frac{2|K|}{\rho_K}
\le 2\kappa\sum_{K\in\mct_h}\frac{|K|}{h_K}
\le\frac{2\kappa}{c_{\rm qu}\,h}\sum_{K\in\mct_h}|K|
=\frac{2\kappa}{c_{\rm qu}}\frac{|\Omega|}{h}.
\end{align}

\textbf{Estimate for $\mathbf{(b)}$:} We use the Cauchy-Schwarz inequality, Young's inequality, and Lemma~\ref{lem:MAC} respectively:
\begin{align}
\label{eq:sum_gamma}
    \text{(b)} &\leq
    \sum_{K\in\mct_h}\int_{K}|\nabla u - \nabla \Pi_h u||\nabla\xi_h|\dx \leq
    \sum_{K\in\mct_h}\int_{K}\frac{1}{2} |\nabla u - \nabla \Pi_h u|^2+\frac{1}{2}|\nabla\xi_h|^2\dx\\&\leq
    \frac{1}{2}C_I^2h^2|u|_{H^2(\Omega)}^2 + \frac{1}{2}|\xi_h|_{H^1(\mct_h)}^2.
\end{align}

Applying bounds (a) and (b) to (\ref{rce:rozdil-presneho-a-priblizneho-reseni-treti-krok}), and using (\ref{eq:sum_gamma}), we get:
\begin{align}
\label{ieq:odhad-z-rozdilu-presneho-a-priblizneho-reseni}
    \frac{1}{2}|\xi_h|_{H^1(\mct_h)}^2\! +\!
    \frac{D}{4}\!\sum_{\Gamma\in\mbe_h}\sum_{a\in e(\Gamma)}\!|\Gamma| [ \xi_h ]^2_a \leq
    \frac{1}{2}C_I^2h^2|u|_{H^2(\Omega)}^2\!+\!\frac{\kappa}{c_{\rm qu}}\frac{|\Omega|}{Dh} \|\nabla u\|_{L^\infty(\mbe)}^2.
\end{align}
Note that by taking $D \geq c_D/h^3$, we get an $\mathcal{O}(h^2)$ estimate of the right-hand side.

Finally, the desired error estimates are obtained by triangle inequalities:
\begin{align*}
    &|u-u_h|_{H^1(\mct_h)}^2 + D\|[ u-u_h ]\|_{L^2(\mbe_h)}^2 \\
    &\leq
    2|u-\Pi_hu|_{H^1(\mct_h)}^2
    +2|\xi_h|_{H^1(\mct_h)}^2
    +2D\|[ u-\Pi_hu ]\|_{L^2(\mbe_h)}^2
    +2D\|[ \xi_h ]\|_{L^2(\mbe_h)}^2.
\end{align*}
The first term is be bounded using Lemma~\ref{lem:MAC}. The third term is zero due to the continuity of $u-\Pi_hu$. The fourth term is estimated using Lemma~\ref{lem:odhad-polynomu-na-hrane}, and (\ref{ieq:odhad-z-rozdilu-presneho-a-priblizneho-reseni}) is applied to the $\xi$-terms. After taking the square root, we get the desired error estimate.
\end{proof}

\begin{proof}[Proof of Theorem~\ref{thm:main}, $L^2$-estimates]
We now proceed with the derivation of the error bounds in the $L^2$-norm using the Aubin-Nitsche `trick'. Since we have $u - u_h \in L^2(\Omega)$ and $\Omega$ is convex polygonal, there exists $\psi \in H^2(\Omega) \cap H^1_0(\Omega)$ such that:
\begin{align}\label{rce:adjoint_problem}
    - \Delta \psi = u-u_h,\quad\text{a.e. in } \Omega, \quad \psi = 0 \,\text{ on }\partial\Omega.
\end{align}

\noindent Moreover, for such $\psi$, we have the following a priori estimate:
\begin{align}\label{ieq:regularity-bound}
    \|\psi\|_{H^2(\Omega)} \leq C_{reg}\|u-u_h\|_{L^2(\Omega)}.
\end{align}

Taking $v \in H^2(\mct_h)$, multiplying the strong form in (\ref{rce:adjoint_problem}), integrating over $K\in \mct_h$, using Green's theorem, and summing over all $K\in\mct_h$, we get:
\begin{align}
    \sum_{K\in\mct_h}\int_K \nabla \psi \cdot \nabla v   \dx - \sum_{\Gamma \in \mbe_h} \int_\Gamma (\nabla\psi \cdot n_\Gamma)[ v]\ds = \int_\Omega (u-u_h)v\dx.\label{rce:adjoint-solution-tested-by-discontinuous-test-function}
\end{align}

Let $\Pi_h\psi \in V_h$ be the continuous piecewise linear Lagrange interpolation of $\psi$ on $\mct_h$. Taking $\Pi_h\psi$ as a test function in both (\ref{thm:2D_DG_scheme:BZ-limitni-TFEM-formulace-suma-pres-vrcholy}) and  (\ref{rce:presne-reseni-testovane-funkci-z-Vh}), respectively, gives us
\begin{align}
    \sum_{K\in\mct_h}\int_K \nabla (u-u_h) \cdot \nabla \Pi_h\psi \dx = 0,\label{temp:orthogonality-to-interpolation}
\end{align}
since $[ \Pi_h\psi ] = [ \Pi_h\psi ]_a = 0$.

\noindent Now setting $v=u-u_h \in H^2(\mct_h)$ in (\ref{rce:adjoint-solution-tested-by-discontinuous-test-function}) and using (\ref{temp:orthogonality-to-interpolation}) gives:
\begin{align}
\label{ieq:L2estimate-beginning}
    \begin{split}
        &\|u-u_h\|^2_{L^2(\Omega)} = \sum_{K\in\mct_h}\int_K \nabla \psi \cdot \nabla(u-u_h)\dx - \sum_{\Gamma\in\mbe_h}\int_\Gamma (\nabla\psi \cdot n_\Gamma)[ u - u_h ] \ds\\ 
        &\ =\sum_{K\in\mct_h}\int_K \nabla (\psi - \Pi_h\psi) \cdot \nabla(u-u_h)\dx - \sum_{\Gamma\in\mbe_h}\int_\Gamma (\nabla\psi \cdot n_\Gamma)[ u - u_h ] \ds\\
        &\ \leq \underbrace{\sum_{K\in\mct_h} |\psi - \Pi_h\psi|_{H^1(K)}|u - u_h|_{H^1(K)}}_{(c)} + \underbrace{\sum_{\Gamma\in\mbe_h} \|[ u - u_h ]\|_{L^2(\Gamma)}\bigg(\int_\Gamma |\nabla\psi \cdot n_\Gamma|^2\ds\bigg)^{1/2}}_{(d)}.
    \end{split}
\end{align}
For the first term: We can use Lemma~\ref{lem:MAC} to estimate $|\psi - \Pi_h\psi|_{H^1(K)}$. For $|u - u_h|_{H^1(K)}$, we can use (\ref{ieq:final-H1-bound}), assuming that $D \geq c_D / h^3$. Finally, we use the a priori estimate (\ref{ieq:regularity-bound}). Altogether we get:
\begin{align}\nonumber
    \begin{split}
        &(c)\leq \bigg(\sum_{K\in\mct_h}|\psi - \Pi_h\psi|_{H^1(K)}^2\bigg)^{1/2}\bigg(\sum_{K\in\mct_h}|u - u_h|^2_{H^1(K)}\bigg)^{1/2}  \\ 
        &\ \leq \bigg(C_I^2h^2\sum_{K\in\mct_h}|\psi|^2_{H^2(K)}\bigg)^{1/2}C_{1}h = C_{1}C_Ih^2|\psi|_{H^2{(\Omega)}} \leq C_{reg}C_{1}C_Ih^2\|u-u_h\|_{L^2(\Omega)},
    \end{split}
\end{align}
\noindent where $C_{1} > 0$ is the constant independent of $h$ from the right-hand side of (\ref{ieq:final-H1-bound}).

For the second term in (\ref{ieq:L2estimate-beginning}), we again use (\ref{ieq:final-H1-bound}) and Lemma~\ref{lem:trace:ineq}: 
\begin{align}
\nonumber
    \begin{split}
        &(d)\leq \frac{1}{\sqrt{D}}\bigg(D\sum_{\Gamma\in\mbe_h} \|[ u - u_h ]\|^2_{L^2(\Gamma)}\bigg)^{1/2}\bigg(\sum_{K\in\mct_h}\sum_{\Gamma \subseteq K}\left\|\frac{\partial\psi}{\partial n_\Gamma}\right\|^2_{L^2(\Gamma)}\bigg)^{1/2}\\
        &\ \leq \frac{1}{\sqrt{D}}C_{1}h\bigg(\sum_{K\in\mct_h}\sum_{\Gamma\subseteq K}\frac{C^2_{Tr}\kappa}{2c_{qu}} h^{-1}\|\psi\|^2_{H^2(K)}\bigg)^{1/2} \\ 
        &\ \leq\frac{1}{\sqrt{c_D/h^3}}C_{1}hC_{Tr}\sqrt{\frac{3\kappa}{2c_{qu}}} h^{-1/2}\|\psi\|_{H^2(K)} \leq C_{1}C_{Tr}\sqrt{\frac{3\kappa}{2c_{qu}c_D}} C_{reg}h^{2}\|u-u_h\|_{L^2(\Omega)}.
    \end{split}
\end{align}
Altogether, plugging the previous two estimates into (\ref{ieq:L2estimate-beginning}) gives us:
\begin{align*}
    \|u-u_h\|^2_{L^2(\Omega)}\leq h^2\|u-u_h\|_{L^2(\Omega)} \bigg(C_{reg}C_{1}C_I + C_{1}C_{Tr}\sqrt{\frac{3\kappa}{2c_{qu}c_D}} C_{reg}\bigg).
\end{align*}

\noindent Dividing both sides by $\|u-u_h\|_{L^2(\Omega)}$, we get the desired $O(h^2)$ estimate.
\end{proof}

\section{Numerical experiments}
\label{sec:Num_exp}
In this section we test the TFEM-DG scheme on problems in 2D and 3D to demonstrate its performance and the sharpness of the theory from Section~\ref{sec:Error_estimates}. \textbf{All the codes are available in the supplementary materials, namely a Gmsh plugin for producing the meshes with dummy elements and the TFEM-DG implementation for the Laplacian in Python.} 

The implementation was performed by following the guideline from page \pageref{Implementation}: (i) transform the mesh by adding zero-measure dummy elements at inter-element interfaces; (ii) put a threshold $J_{\min}\sim h^p,$ $p\ge{d+2},$ on the minimal element Jacobian when computing gradients, element areas, etc., on degenerate elements  in $\R^d$. Therefore (almost) all of the coding work is done on the meshing side which can be completely decoupled from the trivially modified FEM code. The resulting systems were solved with the \emph{PARDISO}~\cite{pardiso} direct solver to avoid sensitivity to iterative solver choices and settings. We note that this is by no means necessary, and conjugate gradients with simple preconditioning performs very well on the resulting systems, similarly as in the case of the original TFEM solver \cite{quiriny2024temperedfiniteelementmethod}. The meshes $\mct_h$ were generated by \emph{Gmsh}~\cite{gmsh}, and then adapted to TFEM-DG by inserting the flat edge dummy elements along edges of $\mct_h$ using the Gmsh plugin provided in the supplementary materials. We used an in-house basic FEM code which was adapted to TFEM-DG by the simple strategy presented on page \pageref{Implementation} (i.e. thresholding near-zero Jacobians to $J_{\min}$). 

\subsection{Optimality of the penalization scaling in 2D}
Theorem~\ref{thm:main} gives a lower bound  $D\gtrsim h^{-3}$ (equivalently $J_{\min}\lesssim h^{4}$ in $\R^2$) in order to get optimal error estimates in $H^1(\mct_h)$ and $L^2(\Omega)$ for the TFEM-DG scheme, i.e. Babu\v{s}ka-Zl\'{a}mal DG with trapezoidal quadrature. Here we test the sharpness of the exponent by taking $J_{\min}=h^{p}$ (equivalently $D=h^{p-1}$) for $p\in \{2.5, 2.8, 3.0, 3.5, 4.0\}$, and plot the resulting convergence curves, cf. Figure~\ref{fig:conv_2D}. Namely, we solve Poisson's problem on the unit square $[0,1]^2$ with the manufactured solution $u(x,y)=\cos(2\pi x)\sin(2\pi x)$. We consider uniform meshes of the type depicted in Figure~\ref{fig:2D_triangulation_FEM_to_DG}. The manufactured solution was used to set Dirichlet boundary conditions and high order quadratures on elements of full measure were used to compute the approximation error.

The results from Figure~\ref{fig:conv_2D} indicate that Theorem~\ref{thm:main} is sharp in the $L^2$-norm, i.e. that $p=4$ is the smallest exponent in $J_{\min}=h^{p}$ giving $\mathcal{O}(h^2)$ convergence in $L^2$. In the $H^1(\mct_h)$-seminorm, it seems that one gets $\mathcal{O}(h)$ convergence even for $p=3$, even though one gets suboptimal $\mathcal{O}(h)$ convergence in $L^2$. We are unable to explain this observation theoretically. Even so, $p=4$ still remains the smallest observed exponent for which one has optimal convergence in both $H^1(\mct_h)$ and $L^2(\Omega)$.

In Figure~\ref{fig:conv_2D_2}, we present convergence curves for $J_{\min}=h^{p}$ with larger $p$, namely $p\in \{4.0, 5.0, 6.0, 7.0\}$. As predicted by the theory, we have optimal convergence rates in all these cases, however for overly large exponents (6.0 and 7.0), and small $h$, the value of $J_{\min}$ nears machine epsilon, which leads to a degradation of the solution. Taking $h=5\cdot 10^{-3}$ and $p=7$ is clearly an overshoot ($J_{\min}=7.8\cdot 10^{-17}$).

In Figure~\ref{fig:2D_fine_mesh}, we present the result of a computation of TFEM-DG computation for the considered problem on a fine mesh with with 811{,}734 degrees of freedom. In this case the mesh is quasi-uniform, but not regular, to demonstrate that the regular right-angled structure is not necessary, cf. the mesh detail in Figure~\ref{fig:2D_fine_mesh}.

\begin{figure}
    \centering
    \includegraphics[width=0.49\textwidth]{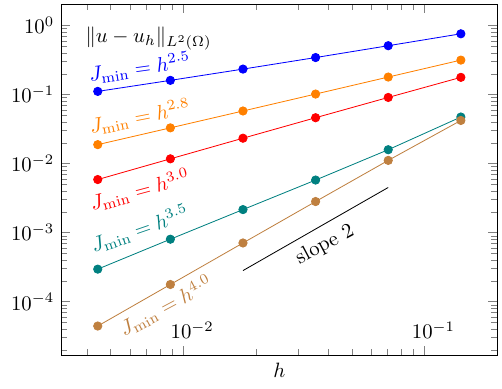}
    \includegraphics[width=0.49\textwidth]{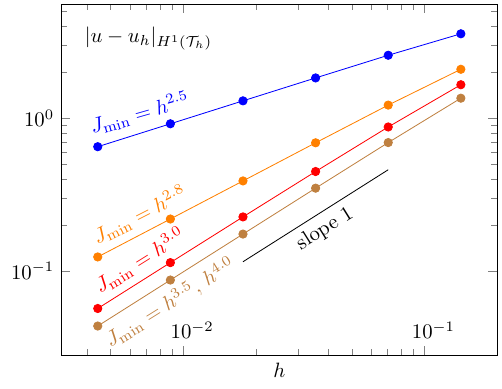}
    \caption{Convergence of TFEM-DG in 2D for various $J_{\min}$ (or equivalently $D$). For $J_{\min} =h^{4.0}$ ($D=h^{-3.0}$) we get  optimal convergence  in both $H^1(\mct_h)$ and $L^2(\Omega)$.}
    \label{fig:conv_2D}
    \includegraphics[width=0.49\textwidth]{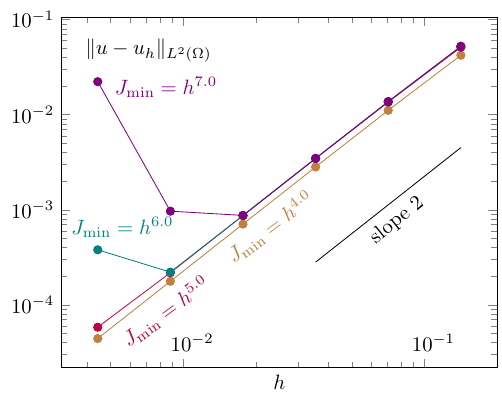}
    \includegraphics[width=0.49\textwidth]{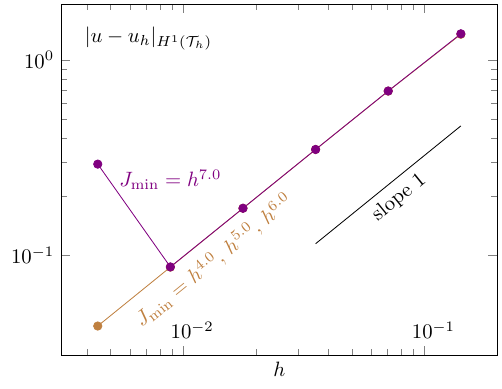}
    \caption{Convergence for $J_{\min}$ lower than $h^{4.0}$. Optimal convergence is preserved, as the theory states. However, some problems arise when $J_{\min}$ nears machine precision.}
    \label{fig:conv_2D_2}
\end{figure}

\begin{figure}
    \centering
    \includegraphics[width=0.4\linewidth]{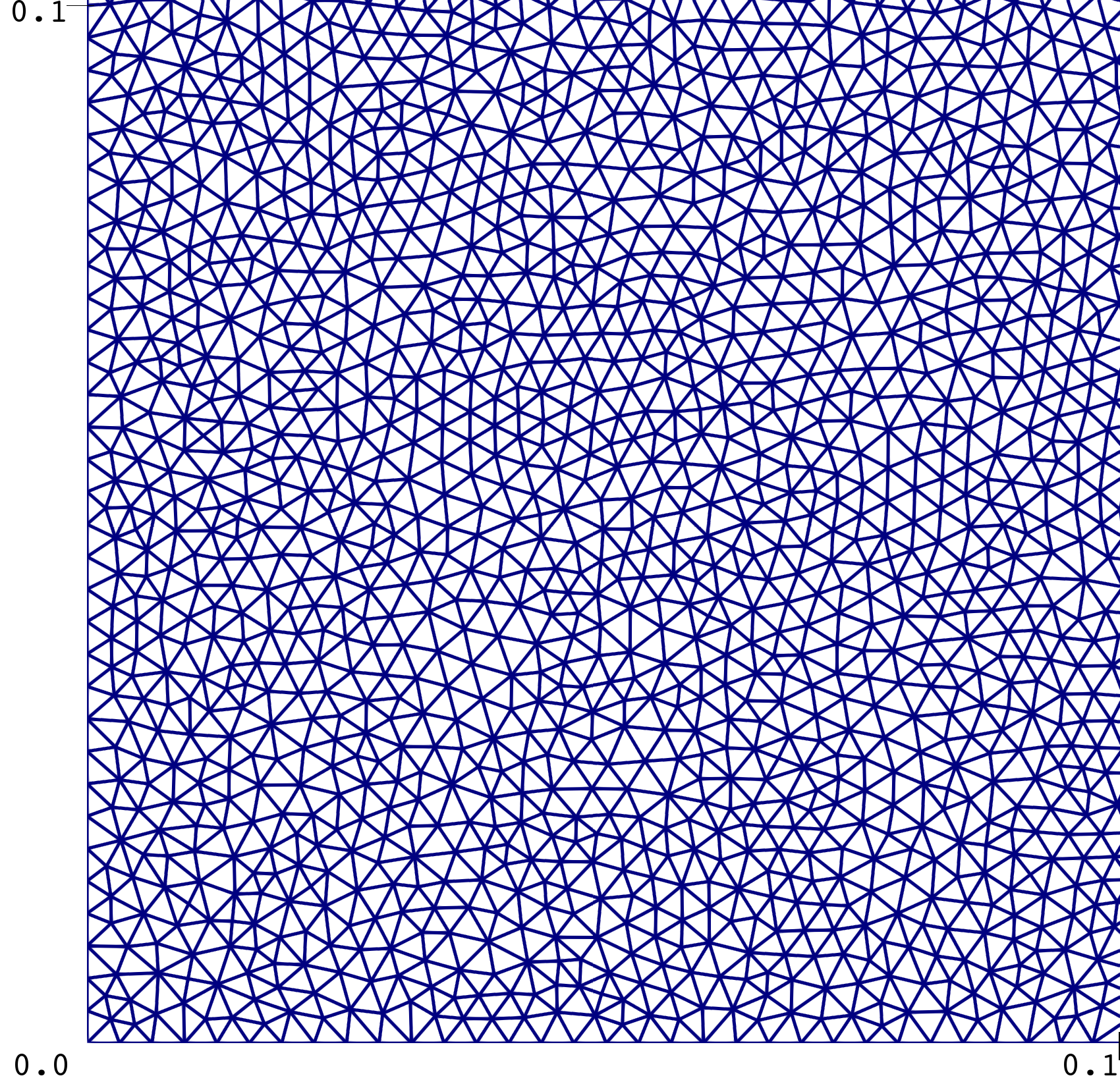}
    \hspace{0.02\linewidth}
    \includegraphics[width=0.4\linewidth]{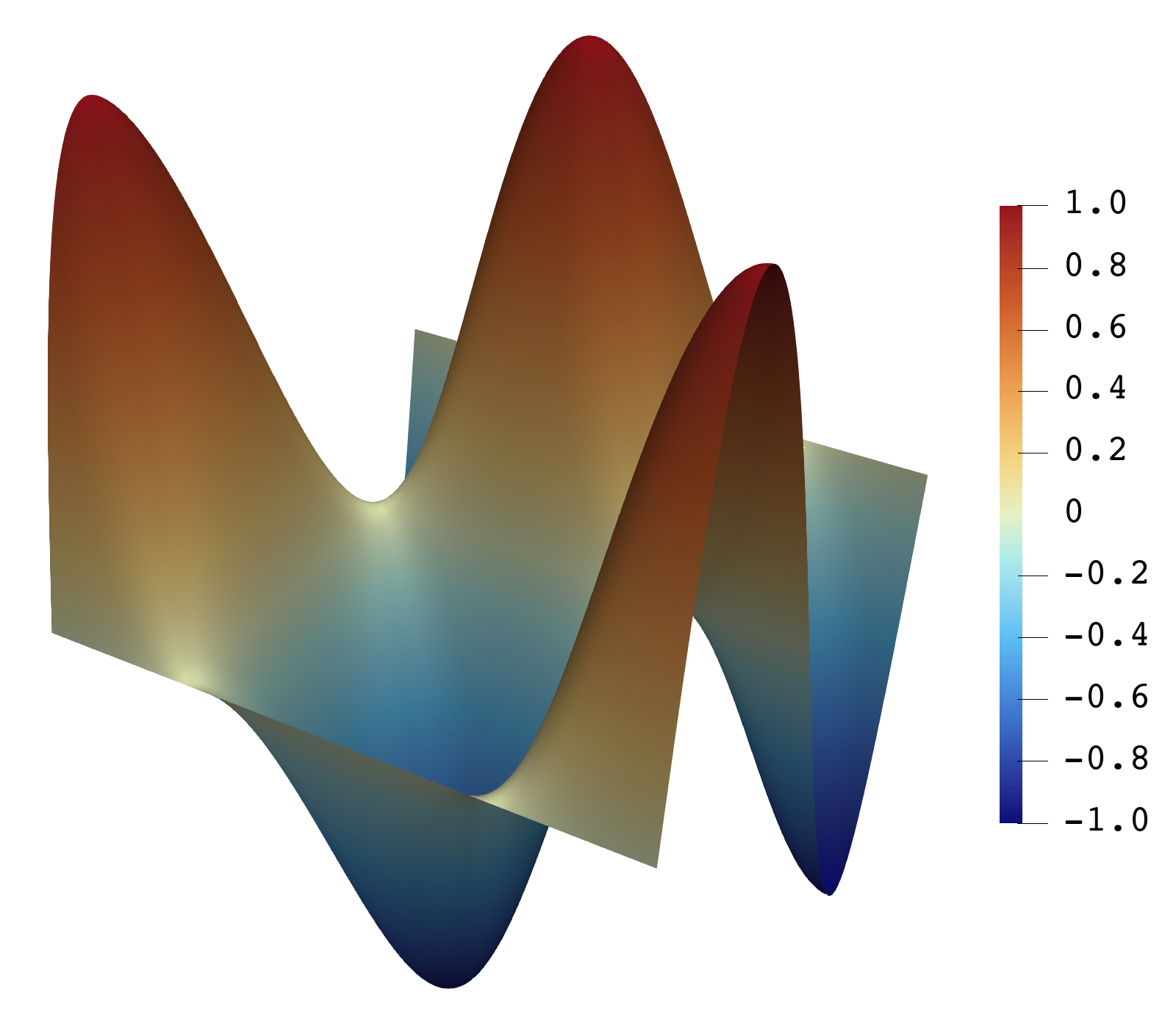}
    \caption{Left: zoomed image of fine unstructured mesh of $[0, 1]^2$. Image is magnified in such a way that the corner $[0, 0.1]^2$ is visible. Right: TFEM-DG solution on this mesh. The computation involved 811{,}734 degrees of freedom.}
    \label{fig:2D_fine_mesh}
\end{figure}

\subsection{Optimality of the penalization scaling in 3D}
Here we extend the results from the previous test case to 3D. Namely, we solve Poisson's problem on the unit cube $[0,1]^3$ with the manufactured solution $u(x,y)=\cos(2\pi x)\sin(2\pi y)\cos(2\pi z)$. Unlike the 2D case, we consider quasi-uniform meshes that are not exactly regular, cf. Figure~\ref{fig:3D_cube}. Figure~\ref{fig:3D_face_tetrahedra} shows the exact configuration of three vanishing face tetrahedra which are inserted between two regular tetrahedra. These dummy face tetrahedra have zero thickness in the actual TFEM-DG computation, however, they are depicted with small non-zero thickness in Figure~\ref{fig:3D_cube} for the sake of visualization. 

Figure~\ref{fig:conv_3D} shows the actual convergence results. In 3D, theory still dictates $D\sim h^{-3}$ as the threshold for optimality. This corresponds to $J_{\min}=h^{5}$ in $\R^3$. We tested $J_{\min}=h^{p}$ (equivalently $D=h^{p-2}$) for $p\in \{3.5, 3.8, 4.0, 4.5, 5.0\}$. Again, the results indicate that we get optimal convergence in both $H^1(\mct_h)$ and $L^2$ for  $J_{\min}=h^{5}$. The slowdown of convergence in the $L^2$-norm for $p=4.5$ is not as prominent as in 2D (for $p=3.5$), however it is still visible. We also seem to get $\mathcal{O}(h)$ convergence in $H^1(\mct_h)$ for the lower exponent $p=4.5$, for $p=4.0$ the situation is not clear.

We also test convergence for larger exponents $p$ in $J_{\min}=h^{p}$, namely $p\in \{5.0, 6.0, 7.0, 10.0\}$, cf. Figure~\ref{fig:conv_3D_2}. As predicted by the theory, we get optimal convergence rates in both $H^1(\mct_h)$ and $L^2(\Omega)$. Unlike in 2D, the slowdown of convergence for extreme exponents ($p=10$) was not observed.

\subsection{Element-wise switching between FEM and DG}
Since the implementation of DG is now done on the mesh level by introducing artificial interface elements, we can choose to do so only in parts of $\Omega$. Namely, we can choose for each edge whether or not to insert these elements, effectively choosing whether or not any given element should be DG or FEM. This is easy to implement based on a simple FEM/DG flag. We present such an example in Figure~\ref{fig:2D_DG-FEM_front_propagation}. There we have a circular front moving through $\Omega$, which gradually switches elements from continuous FEM to DG. This is purely a meshing issue and is completely problem independent! In order to make the transition from FEM to DG visible in Figure~\ref{fig:2D_DG-FEM_front_propagation}, we deliberately take a smaller value of $D=h^{-2}$ to weaken the jump penalization and have more visible jumps in the solution. Also, we only plot the solution on the full-measure elements of $\mct_h$ and omit the vertical segments corresponding to the linear solution on the flat elements, effectively only plotting the true DG part of the solution -- unlike Figure~\ref{fig:1D} in 1D, where we kept the vertical lines to demonstrate that the DG solution comes from FEM. Finally, we display the degenerate edge elements with nonzero thickness in the meshes from Figure~\ref{fig:3D_cube} for the sake of visualization. In the computation itself, these edge elements have area exactly zero.

\begin{figure}
    \centering
    \includegraphics[width=\linewidth]{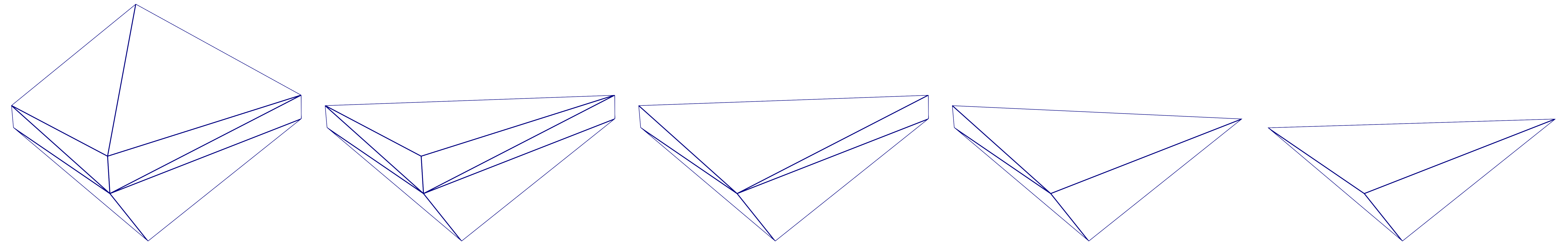}
    \caption{Illustration of how three artificial interface elements (tetrahedra) can be inserted between two full-measure tetrahedra to implement TFEM-DG.}
    \label{fig:3D_face_tetrahedra}
\end{figure}

\begin{figure}
    \centering
    \includegraphics[width=0.8\linewidth]{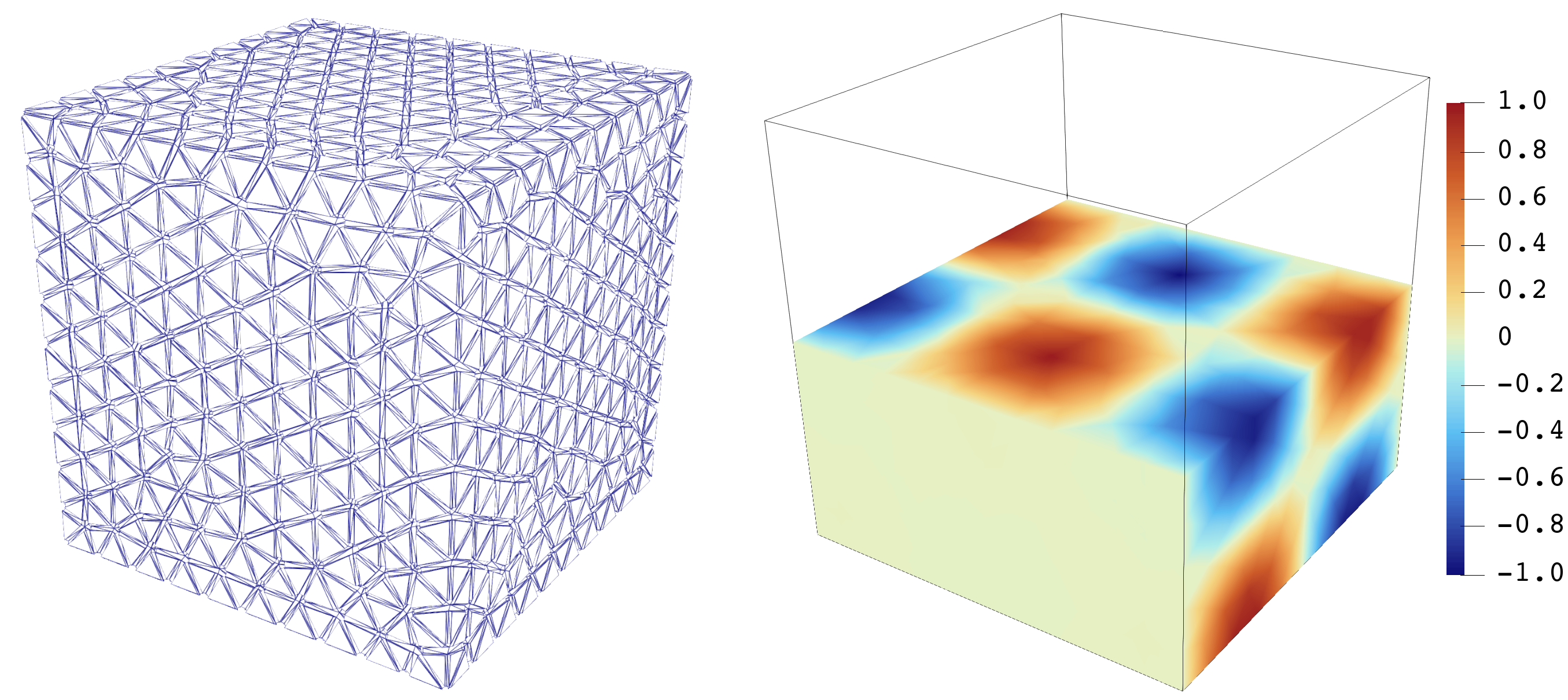}
    \caption{TFEM-DG solution for a problem with the exact solution
    $u(x,y) = \cos(2\pi x)\sin(2\pi y)\cos(2\pi z)$. We set $J_{\min} = h^{5.0}$, i.e. $D = h^{-3.0}$. `Dummy' face tetrahedra are depicted with nonzero thickness for the sake of visualization.}
    \label{fig:3D_cube}
\end{figure}

\begin{figure}
    \centering
    \includegraphics[width=0.49\linewidth]{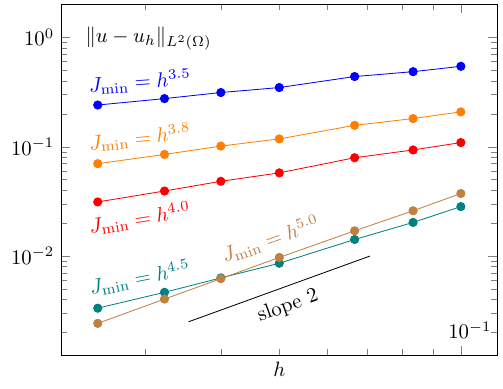}
    \includegraphics[width=0.479\linewidth]{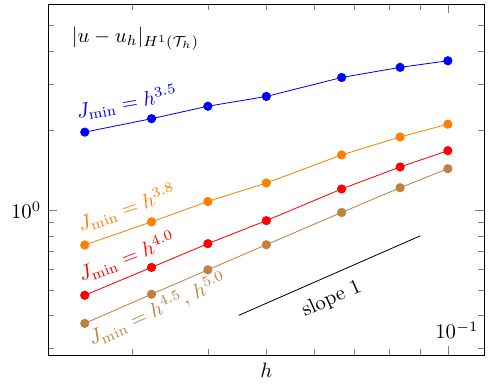}
    \caption{Convergence behavior of TFEM-DG in 3D for various values of $J_{\min}$ (or equivalently $D$). For $J_{\min} =h^{5.0}$ ($D=h^{-3.0}$) we obtain optimal convergence rates.}
    \label{fig:conv_3D}
    
    \includegraphics[width=0.49\linewidth]{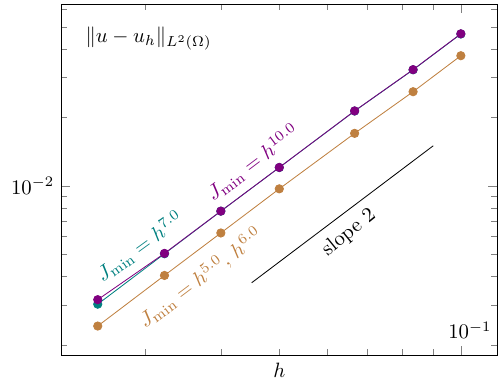}
    \includegraphics[width=0.479\linewidth]{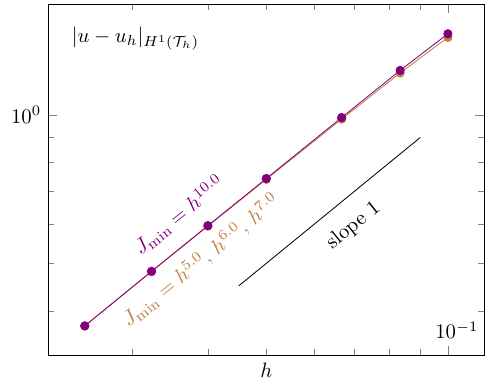}
    \caption{Convergence behavior for values of $J_{\min}$ that are lower than $h^{5.0}$. The order is preserved, as the theory states.}        
    \label{fig:conv_3D_2}
\end{figure}



\begin{figure}
    \centering
    \includegraphics[width=0.25\linewidth]{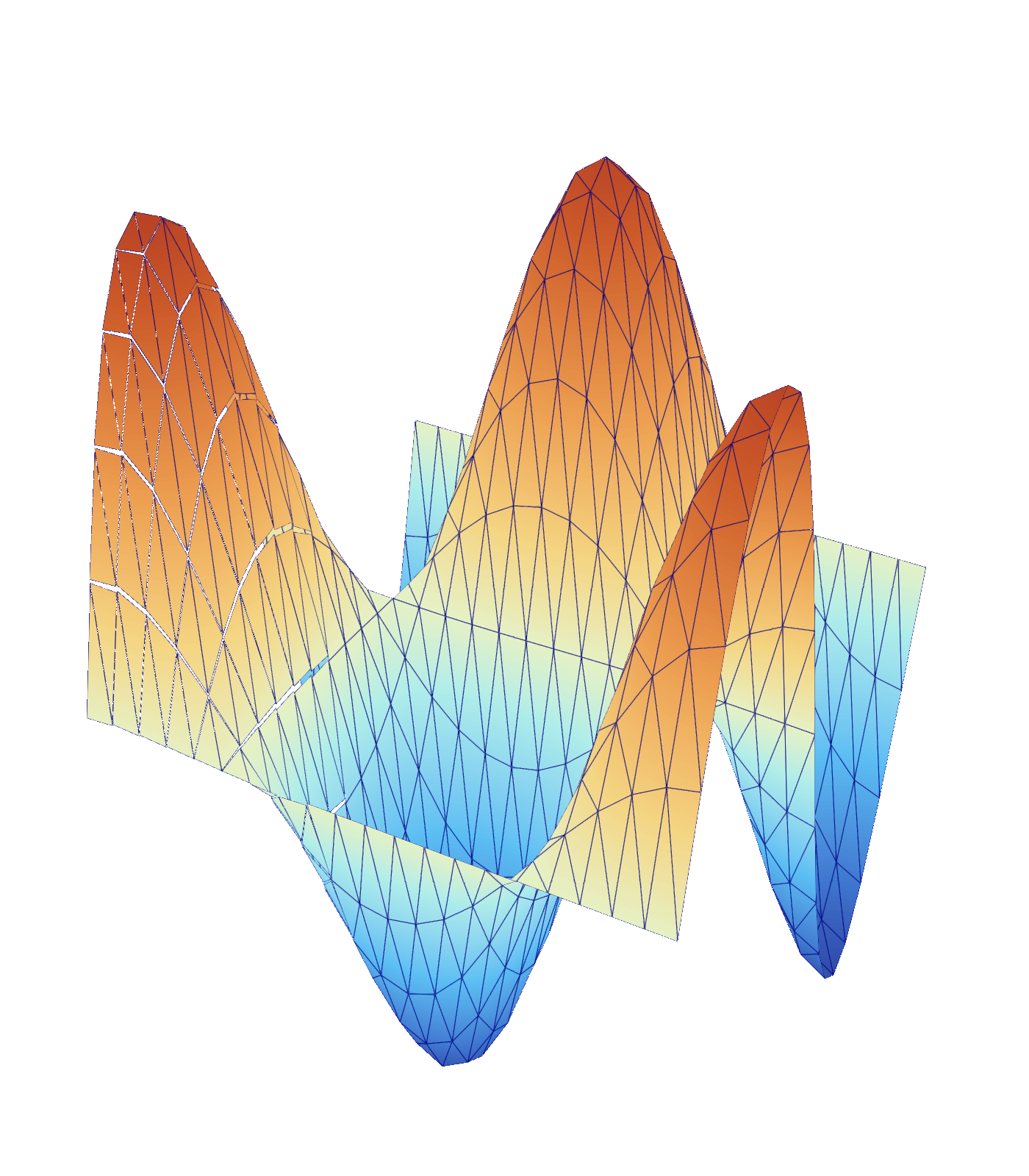}
    \includegraphics[width=0.25\linewidth]{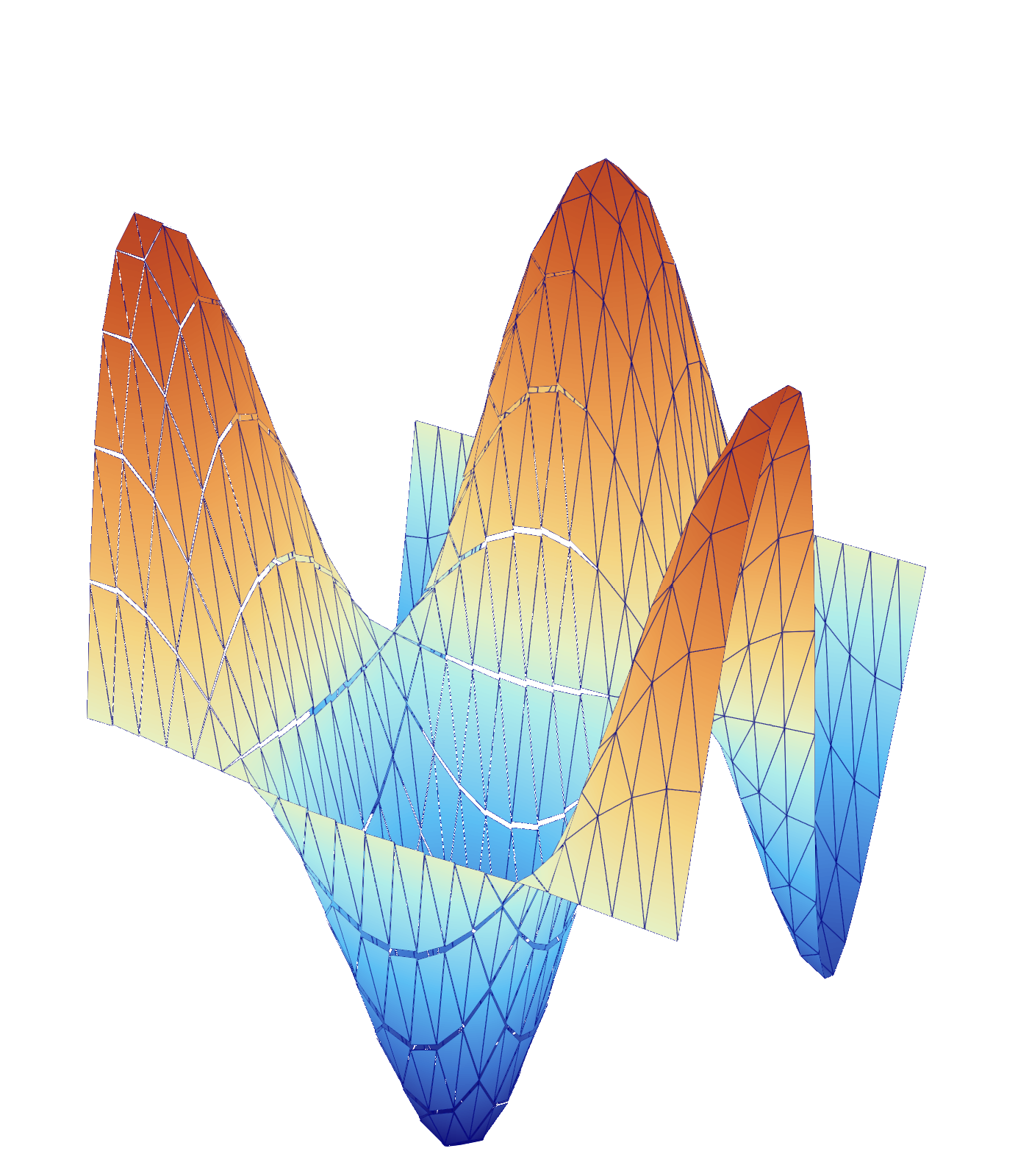}
    \includegraphics[width=0.25\linewidth]{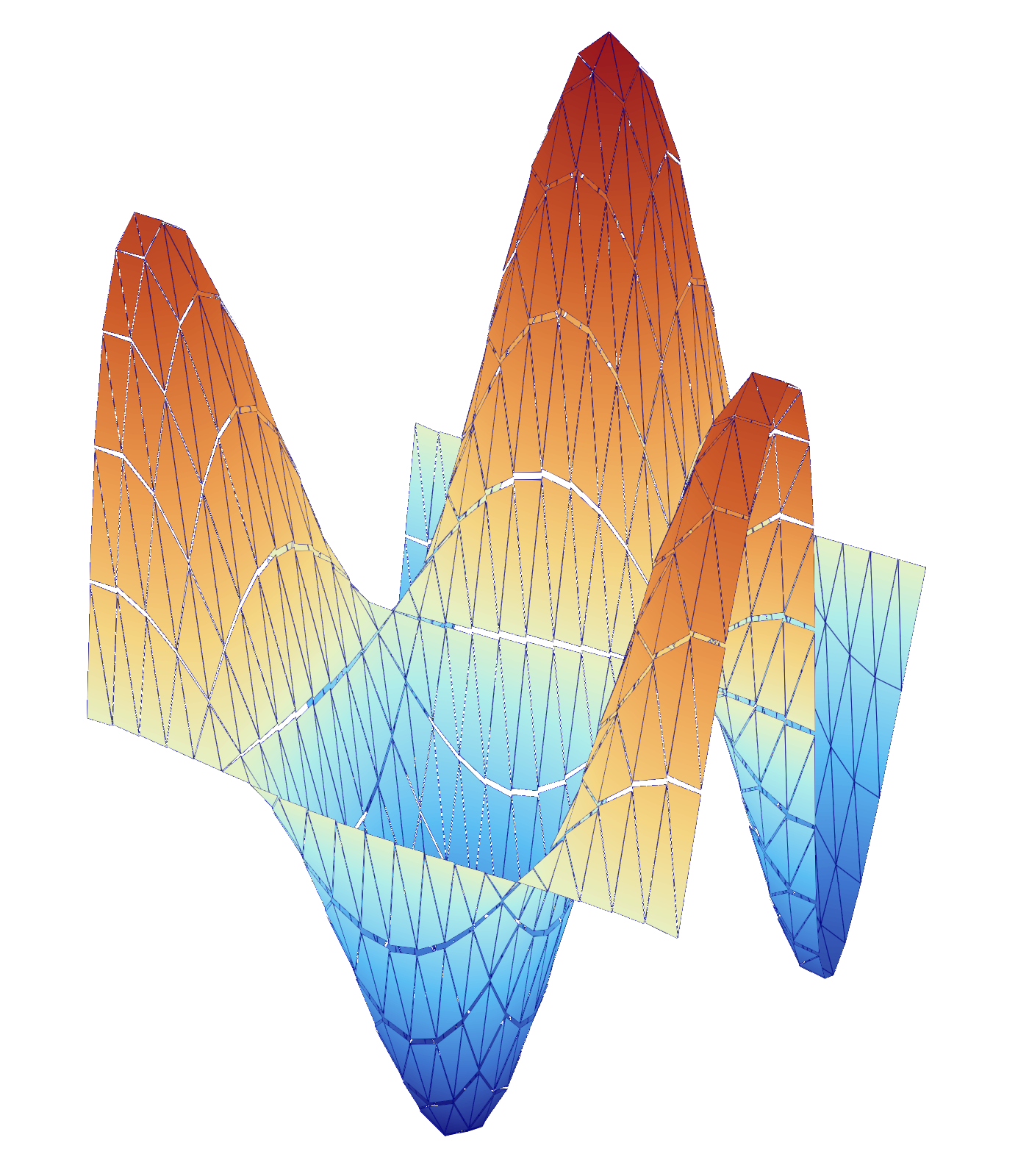}
    \includegraphics[width=0.25\linewidth]{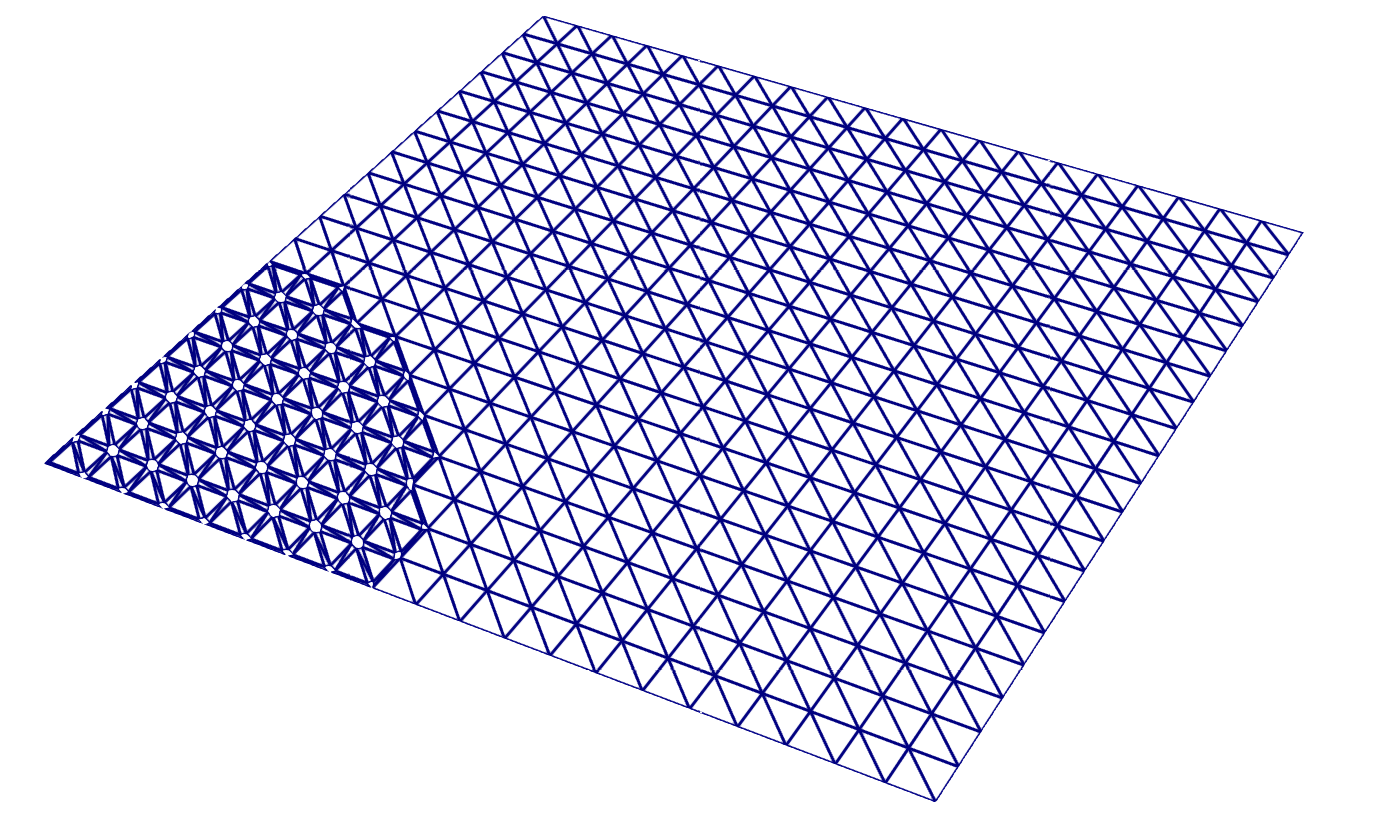}
    \includegraphics[width=0.25\linewidth]{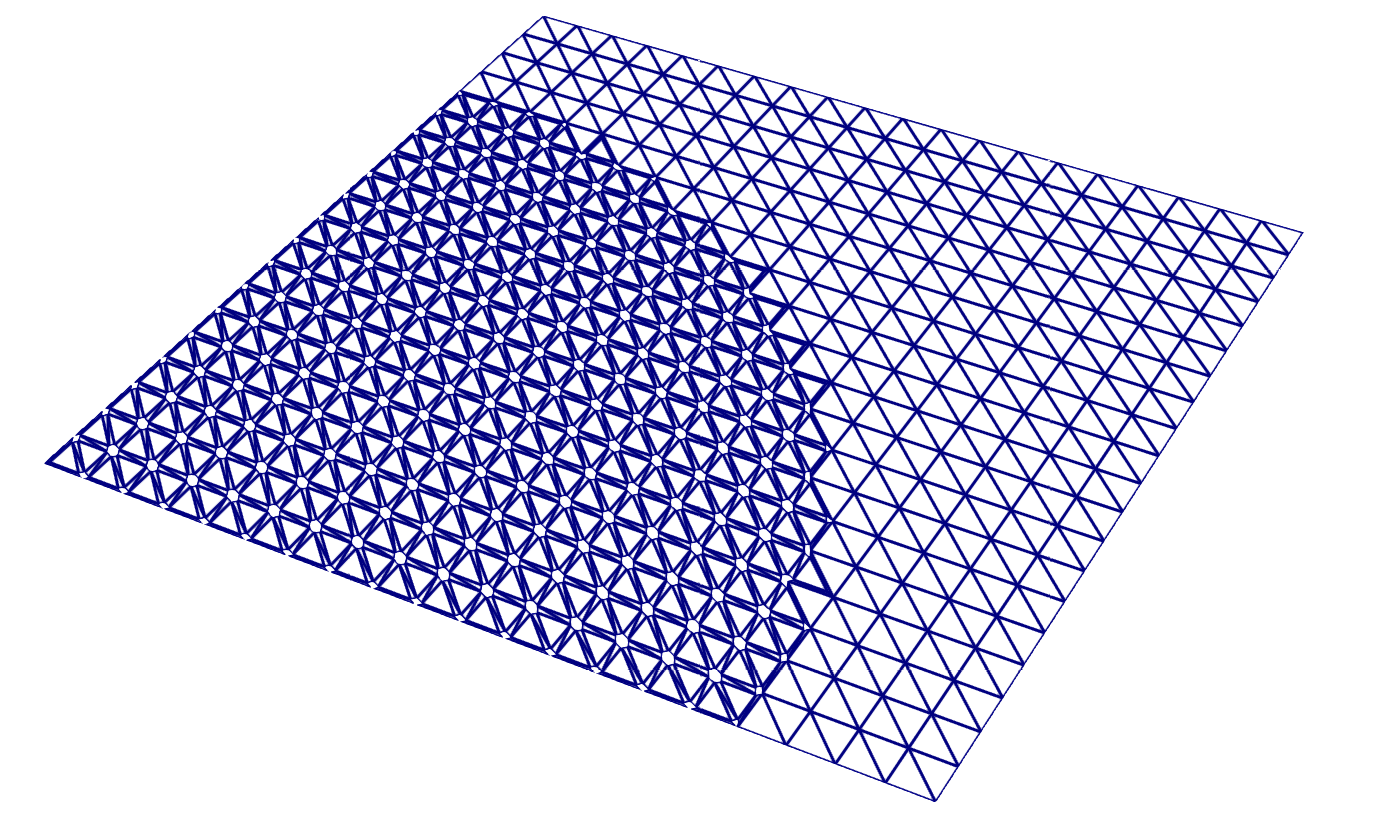}
    \includegraphics[width=0.25\linewidth]{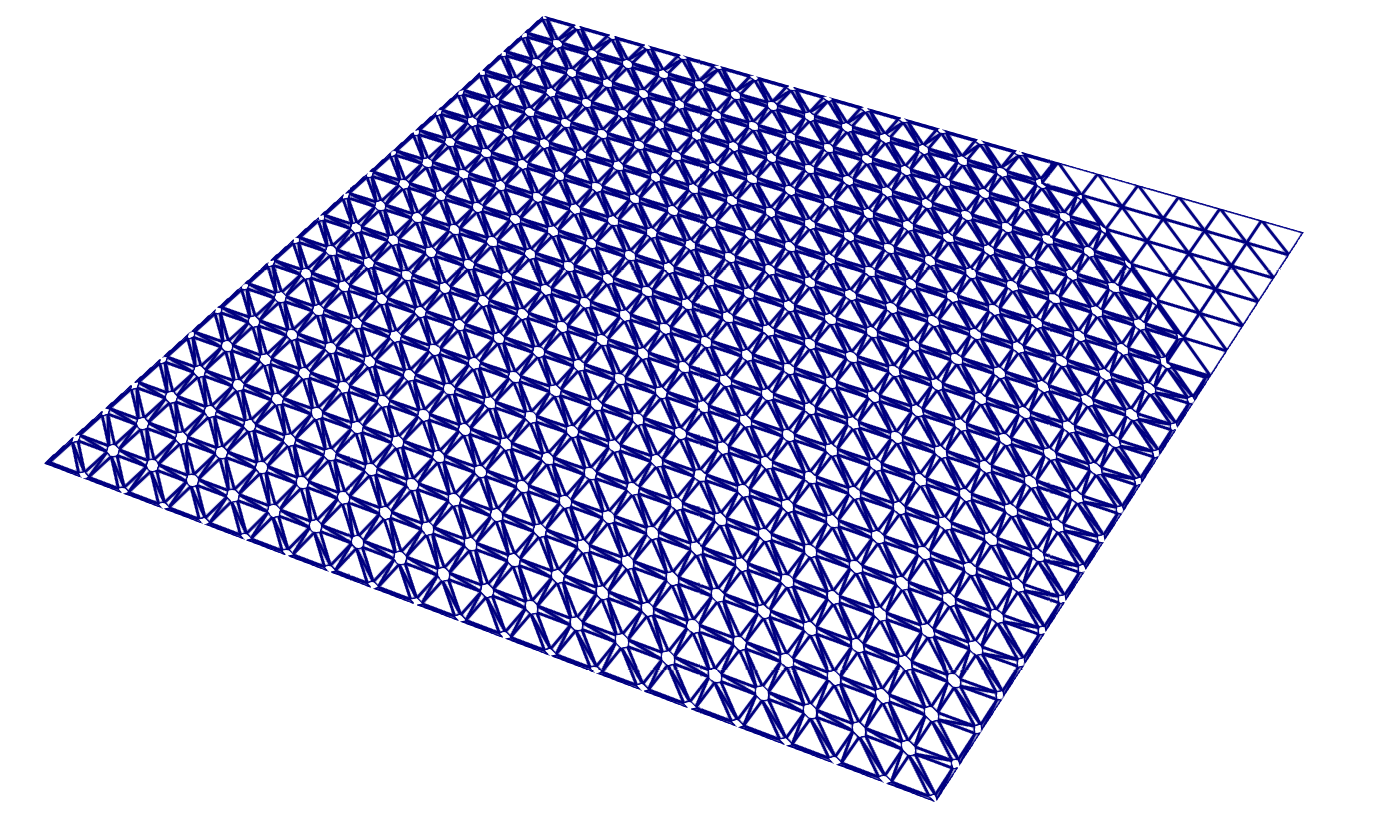}
    \caption{TFEM-DG solution on a series of meshes with a moving circular front which switches individual elements from FEM to DG by inserting (or not inserting) zero-measure edge elements. These dummy elements are depicted with nonzero thickness in the rendered meshes purely for visualization purposes. Also for visualization purposes, we choose a weak penalization coefficient $D = h^{-2}$, in order to see the discontinuity of the numerical solutions in regions where DG is `activated'.}
    \label{fig:2D_DG-FEM_front_propagation}
\end{figure}

\section{Conclusion}
In this work we have established a new connection between classical finite element methods (piecewise linear on simplicial meshes) and the discontinuous Galerkin method. We show that a DG formulation for Poisson's problem can be obtained as the limiting case of finite elements, where certain interface elements are collapsed to zero thickness. Namely, we modify the desired DG mesh by introducing `dummy' elements of width $\beta$ along cell interfaces, and let $\beta \to 0$. This degeneracy is compensated by  modifying the diffusion coefficient to be proportional to $\beta$ on the degenerate face elements. On this modified mesh, we apply standard continuous piecewise linear finite elements. If we let $\beta \to 0$ in 2D, we obtain in the limit the well known Babu\v{s}ka-Zl\'{a}mal DG scheme with trapezoidal quadrature on edges. 

The advantage of this approach is its trivial implementation. The FEM code is modified by introducing a simple threshold $J_{\min}$ for the zero Jacobians of the degenerate elements when computing element integrals, gradients of basis functions, etc., by mapping onto the reference element. The result is a formulation equivalent to the limiting Babu\v{s}ka-Zl\'{a}mal DG scheme. With this trivial modification to any standard FEM code, implementing DG now reduces to modifying the DG mesh by introducing the degenerate dummy element along cell interfaces. This is a simple meshing task, and can be easily implemented completely independently of the FEM code itself and of the problem being solved. Thresholding of the Jacobians is the strategy used in the tempered finite element method (TFEM) developed in \cite{quiriny2024temperedfiniteelementmethod}, which we indicate in the name TFEM-DG.

In the presented paper, we have provided a rigorous derivation of the resulting DG scheme via the limiting procedure. Then we have analyzed the limiting Babu\v{s}ka-Zl\'{a}mal DG scheme with trapezoidal quadrature on edges, proving optimal error estimates in the broken $H^1$-seminorm and the $L^2$-norm. The analysis gives a `recipe' for the value of the jump penalization parameter, or equivalently, for the threshold $J_{\min}$ used in the implementation. Finally, we test the method on Poisson problems with manufactured solutions in 2D and 3D, showing optimality of the theory with respect to $H^1$ and $L^2$ convergence rates and the choice of the parameter $J_{\min}$.

Altogether, the presented framework allows the easy implementation of DG in an off-the-shelf FEM code by making a simple tweak in the code, and modifying the computational mesh. The mesh modifications allows to easily switch individual elements (or edges) of the mesh to be either continuous FEM, or DG, according to the user's choice. We demonstrate this on a problem with a moving front which gradually switches FEM to DG. 

The presented framework can be easily extended to other problems. In a followup paper, we will demonstrate this for evolutionary nonlinear first order hyperbolic equations and their systems, such as the compressible Euler equations. There the degenerate edge elements naturally give rise to a numerical flux.

We believe the benefit of the presented approach is the possibility to easily switch between FEM and DG, even element-wise. This could be motivated by the problem, where in part of the computational domain we have, e.g., the heat equation or elasticity, while in another part we solve compressible Euler equations. While FEM is suitable for the first, DG is suitable for the latter, and TFEM-DG allows us to easily implement the solver in a monolithic fashion in a FEM code, without the need to implement a combined FEM-DG solver. Another motivation might be user-driven, when the user who has implemented finite elements for his/her specific PDE now wishes to test if discontinuous Galerkin would perform better. The TFEM-DG framework, allows to easily test DG within the original FEM code without having to re-implement the problem in a separate DG code.

By turning DG into a limit of FEM, TFEM‑DG reduces implementation to a mesh edit and a simple thresholding tweak, enabling practical and rigorous switching between FEM and DG wherever the physics or user demands. Discontinuity thus becomes a simple modeling choice rather than a software rewrite.


\bibliographystyle{plain}
\bibliography{bib}

\begin{thebibliography}{10}

\bibitem{arnold-et-al-unified-analysis-of-dg-for-eliptic-problems}
{D}ouglas {A}rnold, {F}ranco {B}rezzi, {B}ernardo {C}ockburn, and {L}. {M}arini.
\newblock {U}nified {A}nalysis of {D}iscontinuous {G}alerkin {M}ethods for {E}lliptic {P}roblems.
\newblock {\em SIAM J. Numer. Anal.}, 39:1749--1779, 07 2006.

\bibitem{Babuska:1973:FEM_penalty}
Ivo Babu{\v{s}}ka.
\newblock The finite element method with penalty.
\newblock {\em J. Math. Comput.}, 27(122):221--228, 1973.

\bibitem{Babuska-Zlamal}
{I}vo {B}abu\v{s}ka and Milo\v{s} Zl\'{a}mal.
\newblock Nonconforming elements in the finite element method with penalty.
\newblock {\em {SIAM} {J}. {N}umer. {A}nal.}, 10(5):863–875, 1973.

\bibitem{Barenblatt1962}
Grigory~I. Barenblatt.
\newblock The mathematical theory of equilibrium cracks in brittle fracture.
\newblock {\em Adv. Appl. Mech.}, 7:55--129, 1962.

\bibitem{WOPSIP}
Susanne Brenner, Luke Owens, Li-Yeng And, and Li-yeng Sung.
\newblock A weakly over-penalized symmetric interior penalty method.
\newblock {\em Electron. Trans. Numer. Anal.}, 30:107 -- 127, 01 2008.

\bibitem{ChenShuEntropyStableReview}
Tianheng Chen and Chi-Wang Shu.
\newblock Review of entropy stable discontinuous {G}alerkin methods for systems of conservation laws on unstructured simplex meshes.
\newblock {\em CSIAM Trans. Appl. Math.}, 1(1):1–52, Apr. 2020.

\bibitem{Ciarlet}
Philippe~G. Ciarlet.
\newblock {\em The Finite Element Method for Elliptic Problems}.
\newblock Society for Industrial and Applied Mathematics (SIAM), Philadelphia, 2002.

\bibitem{CockburnShu1998_V}
Bernardo Cockburn and Chi-Wang Shu.
\newblock The {R}unge--{K}utta discontinuous {G}alerkin method for conservation laws v: {M}ultidimensional systems.
\newblock {\em J. Comput. Phys.}, 141(2):199--224, 1998.

\bibitem{deFranciscoCarol2020}
Miguel de~Francisco and Ignacio Carol.
\newblock Displacement-based and hybrid formulations of zero-thickness mortar/interface elements for unmatched meshes, with application to fracture mechanics.
\newblock {\em Int. J. Numer. Methods Eng.}, 2020.

\bibitem{Dugdale1960}
Donald~S. Dugdale.
\newblock Yielding of steel sheets containing slits.
\newblock {\em J. Mech. Phys. Solids}, 8(2):100--104, 1960.

\bibitem{DurandTrinidade2021}
Rodrigo Durand and Felipe H. B.~Trinidade da~Silva.
\newblock Three-dimensional modeling of fracture in quasi-brittle materials using plasticity and cohesive finite elements.
\newblock {\em Int. J. Fract.}, 228:45--70, 2021.

\bibitem{GaroleraEtAl2013}
Daniel Garolera, Ignasi Aliguer, J.~M. Segura, Ignacio Carol, M.~R. Lakshmikantha, and J.~Alvarellos.
\newblock Zero-thickness interface elements with h--m coupling: formulation and applications in geomechanics.
\newblock In {\em XII International Conference on Computational Plasticity (COMPLAS XII)}, pages 1--12. CIMNE, 2013.

\bibitem{gmsh}
Christophe Geuzaine and Jean-François Remacle.
\newblock Gmsh: A 3-{D} finite element mesh generator with built-in pre- and post-processing facilities.
\newblock {\em Int. J. Numer. Methods Eng.}, 79:1309 -- 1331, 09 2009.

\bibitem{EsfahaniGajo2024}
Farzaneh Ghalamzan~Esfahani and Alessandro Gajo.
\newblock A zero-thickness interface element incorporating hydro--chemo--mechanical coupling and rate-dependency.
\newblock {\em Acta Geotech.}, 19:197--220, 2024.

\bibitem{pardiso}
{I}ntel {C}orporation.
\newblock {PARDISO} - {P}arallel {D}irect {S}parse {S}olver {I}nterface.
\newblock \url{https://www.intel.com/content/www/us/en/docs/onemkl/developer-reference-c/2023-0/pardiso.html}.

\bibitem{ParkPaulino2011}
Kyoungsoo Park and Glaucio~H. Paulino.
\newblock Cohesive zone models: A critical review of traction--separation relationships across fracture surfaces.
\newblock {\em Appl. Mech. Rev.}, 64(6):060802, 2011.

\bibitem{quiriny2024temperedfiniteelementmethod}
Antoine Quiriny, Václav Kučera, Jonathan Lambrechts, Nicolas Mo\"{e}s, and Jean-François Remacle.
\newblock The tempered finite element method.
\newblock {\em J. Comput. Phys.}, 549, 2026.

\bibitem{ReedHill1973}
William~H. Reed and Thomas~R. Hill.
\newblock Triangular mesh methods for the neutron transport equation.
\newblock Technical Report LA-UR-73-479, Los Alamos Scientific Laboratory, 1973.

\bibitem{Shu2013BriefSurvey}
Chi-Wang Shu.
\newblock A brief survey on discontinuous {G}alerkin methods in computational fluid dynamics.
\newblock {\em Adv. Mech.}, 43(6):541--553, 2013.

\end{thebibliography}

\end{document}